\documentclass{article}

\usepackage{amsmath, amssymb, epsfig, eepic, bm, amsbsy, color}

\newcommand{\bt}   {\begin{theorem}}
\newcommand{\et}   {\end  {theorem}}
\newcommand{\bl}   {\begin{lemma}}
\newcommand{\el}   {\end  {lemma}}
\newcommand{\bp}   {\begin{prop}}
\newcommand{\ep}   {\end  {prop}}
\newcommand{\bc}   {\begin{cor}}
\newcommand{\ec}   {\end  {cor}}
\newcommand{\bd}   {\begin{defn}}
\newcommand{\ed}   {\end  {defn}}

\newcommand{\ba}   {\begin{array}}
\newcommand{\ea}   {\end  {array}}
\newcommand{\be}   {\begin{enumerate}}
\newcommand{\ee}   {\end  {enumerate}}
\def\eq#1\en{\begin{equation}#1\end{equation}}
\def\eqsplit#1\ensplit{
    \begin{equation}\begin{split}#1\end{split}\end{equation}
    }
\def\eqalign#1\enalign{
    \begin{align}#1\end{align}
    }
\def\eqmul#1\enmul{
    \begin{multline}#1\end{multline}
    }
\newcommand{\eqarrstar} {\begin{eqnarray*}}
\newcommand{\enarrstar} {\end{eqnarray*}}
\newcommand{\eqarray}   {\begin{eqnarray}}
\newcommand{\enarray}   {\end{eqnarray}}

\makeatletter
\newcommand{\labelcounter}[2]{{%
    \stepcounter{#1}
    \protected@write\@auxout{}%
    {\string\newlabel{#2}{{\csname the#1\endcsname}{\thepage}}}%
    {\ref{#2}}
    }}
\makeatother
\newcommand{\sss}   { \scriptscriptstyle }

\newcommand{\spose}[1] {{\hbox to 0pt{#1\hss}} }
\newcommand{\ltapprox} {\mathrel{\spose{\lower 3pt\hbox{$\mathchar"218$}}
 \raise 2.0pt\hbox{$\mathchar"13C$}}}
\newcommand{\gtapprox} {\mathrel{\spose{\lower 3pt\hbox{$\mathchar"218$}}
 \raise 2.0pt\hbox{$\mathchar"13E$}}}

\newcommand{\bra}[1]    {\left \langle #1 \right |}

\newcommand{\namedref}{\S\ref}

\numberwithin{equation}{section}

\newcommand{\tprob}{\mathbb P}
\newcommand{\mprob}[1]{\mathbb P\br{#1}}

\newcommand{\mn}[1]{\|#1\|_{\sigma}}
\newcommand{\mexpec}[1]{\mathbb E\hspace{-.2em}\brh{#1}}

\newcommand{\mcexpec}[2][\tau]{\hat {\mathbb E}_{#1}\hspace{-.3em}\brh{#2}}
\newcommand{\msprob}[2][\sigma]{\mathbb P_{\hspace{-.15em} #1} \hspace{-.20em} \br{#2}}
\newcommand{\mcprob}[2][\tau]{\hat{\mathbb P}_{\hspace{-.15em} #1} \hspace{-.20em} \br{#2}}
\newcommand{\eqn}[1]{\begin{equation} #1 \end{equation}}
\newcommand{\indic}[1]{{{\bf 1}{\scriptstyle  \bra{#1}}}} 
\newcommand{\br}[1]{\left( #1 \right)}       
\renewcommand{\bra}[1]{\left\{ #1 \right\}}  
\newcommand{\brh}[1]{\left[ #1 \right]}    
\newcommand{\brc}[3]{\left#1 #3 \right#2}  
\newcommand{\mO}[1]{{\cal O}\br{#1}}    

\newcommand{\qed}    {\hfill$\Box$}
\newcommand{\smallsup}[1] {{\scriptscriptstyle{(\kern-.03em{#1}}\kern-.03em)}}

\newtheorem{thm}{Theorem}[section]

\newtheorem{lem}[thm]{Lemma}
\newtheorem{prop}[thm]{Proposition}
\newtheorem{cor}[thm]{Corollary}

\newtheorem{rem}[thm]{Remark}

\begin{document}		
\title{A geometric preferential attachment model \\ with fitness}
	\date{3 January 2008}
	\author{H. van den Esker
	\thanks{Delft University of Technology,
Electrical Engineering, Mathematics and Computer Science, P.O. Box
5031, 2600 GA Delft, The Netherlands. E-mail: {\tt
H.vandenEsker@tudelft.nl}}}
	\maketitle
	
	\begin{abstract}
		\noindent We study a random graph $G_n$, which  combines aspects of geometric random graphs and preferential attachment. 
		The resulting random graphs have power-law degree sequences with finite mean and possibly infinite variance. In particular, the power-law exponent can be any value larger than 2.
		
		The vertices of $G_n$ are $n$ sequentially generated vertices chosen at random in the unit sphere in $\mathbb R^3$. 
		A newly added vertex has $m$ edges attached to  it and the endpoints of these edges are connected to old vertices or to the added vertex itself.
		The vertices are chosen with probability	proportional to their current degree plus some initial attractiveness and multiplied by a function, depending on the geometry.
	\end{abstract}
	
	\begin{flushleft}
	
	\bigskip
	
	\end{flushleft}

		
			\newcommand{\mproef}[2]{#1}	
	
		\newcommand{\mR}{A_{\sigma,n}}
	\newcommand{\head}[1]{ \textsc{to}(#1)}	
	\newcommand{\mX}[2][\sigma]{{ #2}_{#1}}
	\newcommand{\mhX}[2][\sigma]{ \hat { #2}_{#1}}
	\newcommand{\mN}[2][\sigma]{N^{#2}_{#1}}
	\newcommand{\mXS}[3]{#3_{#1}^{\sss (#2)}}
	\newcommand{\mv}[2]{\mXS{#1}{#2}{v}}
	\newcommand{\mb}[2]{\mXS{#1}{#2}{b}}
	\newcommand{\my}[2]{\mXS{#1}{#2}{y}}
	\newcommand{\mhv}[2]{\mXS{#1}{#2}{\hat v}}
	\newcommand{\mhb}[2]{\mXS{#1}{#2}{\hat b}}
	\newcommand{\mhy}[2]{\mXS{#1}{#2}{\hat y}}

	\section{Introduction} \label{sec:intrduction}	
	Preferential attachment models are proposed by Barab\'asi and Albert \cite{barabasi-1999-286} as models for large-scale networks like the Internet, electrical networks, telephone networks, and even complex biological networks.
	These networks grow in time, because, for example, new routers, transform houses, switchboards or proteins are added to the network.
	The behavior can be modeled by means of a random graph process.
		A random graph process is a stochastic process that describes a random graph evolving with time.
		At each time step, the random graph is updated using a given \emph{rule of growth}, which will be specified later.
	
	In literature a number of different rules of growth are explored.
	For example, each time step we add or remove edges/vertices \cite{ccooafrie01}, or, more advanced, copy parts of the graph \cite{kumar00stochastic}. 
	Furthermore, there is freedom in the choice how to connect endpoints of newly added edges.	
	Mostly, one randomly chooses the endpoint over the vertices, or proportional to the degree. 
	Another possibility is to assign to each vertex a \emph{fitness}. 
	In \cite{parid?, PhysRevLett.85.4633} \emph{additive fitness} is explored
	where one chooses proportional to the degree plus some (random) value. 
	In  \cite{ergun-2002-303, norros3} \emph{multiplicative fitness} is explored where each vertex has a random fitness and one chooses a vertex proportional to the degree times the fitness. 
	In this paper we use a constant additive fitness and a variant of multiplicative fitness, depending on the distance between vertices.
	
	Many large networks of interest have power-law degree sequences, by which we mean that the number of vertices with degree $k$ falls off as $k^{-\tau}$ for some exponent $\tau>1$.
	The parameter $\tau$ is called the power-law exponent.
		Depending on the value of $\tau$ we classify the following three categories: the \emph{infinite mean case}, \emph{the finite mean and infinite variance case}, and the \emph{finite variance case}, which corresponds to $\tau \in(1,2)$,  $\tau \in(2,3)$ and $\tau>3$, respectively.
		
		These categories are of interest, because the behavior of the typical distance is determined by the power-law exponent $\tau$.
		Results in the literature show that if $\tau \in(1,2)$ the typical distance is bounded by some constant, if $\tau\in(2,3)$ the typical distance is concentrated around $\log\log n$ and if $\tau>3$ it is concentrated around $\log n$, where $n$ is the number of vertices of the graph, see \cite{infmean, diaminpref, finstub, infvar, norros3}.

	A large number of graph models have been introduced to describe complex networks, 
	but often the underlying geometry is ignored. 
	In general it is difficult to get rigorous results for properties like the degree distribution, typical distances or diameter, even if one disregards the geometry. 		
	However, in wireless ad-hoc networks the geometry is of great importance,
	since in these networks nodes are spread over some surface and nodes can only communicate with neighbors within a certain range, depending on the geometry.

	In this paper we will rely on the geometric preferential attachment (GPA) model introduced in \cite{flax04} and extended in \cite{flax07} by the same authors. 
The GPA model is a variant of the well known Barab\'asi-Albert (BA) model. 
	In the BA model new vertices are added to the graph one at a time. 
	Each new vertex is connected to $m$ of the existing vertices, 
	where we choose to connect to an old vertex proportional to its degree.
	In the GPA model each vertex has a position on a surface, and we choose to connect to an old vertex proportional to its degree times a non-constant multiplicative value.
	This multiplicative value depends on the distance of the old vertex and the newly added vertex.
	For instance, let the multiplicative constant be 1 if the vertices are at distance at most $r_n$, and otherwise zero. 
	The latter attachment rule essentially describes the construction of a simplified wireless ad-hoc network.

	\subsection{Definition of the model} \label{gpaf:sec:model}			
	In this section we will introduce the Geometric Preferential Attachment model with fitness (GPAF).
	The GPAF model is described by a random graph process $\{G_\sigma\}_{\sigma=0}^n$, which we will study for large values of $n$. 
	For $0 \leq \sigma \leq n$, each vertex of the graph $G_\sigma=(V_\sigma,E_\sigma)$ is positioned on the sphere $S \subset \mathbb R^3$.
	The radius of the sphere $S$ is taken equal to $1/(2\sqrt{\pi})$, so that, conveniently, $\text{Area}(S)=1$. 
	The vertices of the graph $G_\sigma$ are given by $V_\sigma=\{1,2,\ldots,\sigma\}$ and $E_\sigma$ is the set of edges. 
	The position of vertex $v\in V_\sigma$ in the graph $G_\sigma$ is given by $x_v\in S$ and the degree at time $\sigma$ is given by $d_\sigma(v)$.

	In total we need 4 parameters to describe the GPAF model.
	The first parameter of the model is $m=m(n)>0$, which is the number of edges added in every time step. 
	The second parameter is $\alpha \geq 0$, which is a  measure of the bias toward self-loops. 
	The third one is $\delta > -m$, which is the initial attractiveness of a vertex. 
	And, finally, the fourth parameter is a function $F_n:[0,\pi]\rightarrow {\mathbb R}_+$, where the value $F_n(u)$ is a indicator of the attraction between two vertices at distance $u$.

	Before we give the model definition, we first will explain the use of the parameter $\alpha$.
	Assume that the graph $G_\sigma$ is given, consisting of the vertices $V_\sigma$. 
	We construct the graph $G_{\sigma+1}$ by choosing vertex $x_{\sigma+1}$ uniformly at random in $S$ and add it to $G_\sigma$ with $m$ directed edges emanating from the vertex $x_{\sigma+1}$. 
			Let, for $\sigma=1,2,\ldots,n-1$, 
			\eqn{\label{eq:Tt}
			 T_{\sigma,n}(u)=\sum_{v=1}^\sigma (d_\sigma(v) +\delta) F_n(|x_v-u|),
			}
			where $|x_v-u|\in[0,\pi]$ is the angular distance from $u$ to $u_0$ along a circle with radius $1/(\sqrt{2})$ over the sphere $S$.
			Furthermore, let the endpoints of the $m$ emanating edges be given by the vertices $\mv{\sigma+1}1,\ldots, \mv{\sigma+1}m$.
			Intuitively, we would like to choose the endpoints at random (with replacement) from $V_\sigma$, such that $v \in V_\sigma$ is chosen with probability
			\begin{align*}
				\msprob[\sigma]{\mv{\sigma+1}i=v}&=\frac{(d_{\sigma}(v)+\delta)F_n(|x_v-x_{\sigma+1}|)}{T_{\sigma,n}(x_{\sigma+1})},&
			\end{align*}
		where $\msprob[\sigma]{\,\cdot\,}=\mprob{\,\cdot\, | G_\sigma, x_{\sigma+1}}$.		
	However, the above given rule of growth is not well-defined. To see this, consider the simplified model for wireless ad-hoc networks, i.e., $F_n(x)=\indic{x \leq r_n}$. Then, for any $\sigma$, there is a positive probability that there are no vertices within reach of the newly added vertex $x_{\sigma+1}$ and therefore $T_{\sigma,n}(x_{\sigma+1})=0$.
	
	Introducing self-loops solves this problem and for this the additional paramater $\alpha$ is introduced. We will follow the solution given by the authors of \cite{flax07} for the GPA model:
		
		\bigskip
		
	\noindent {\bf Rules of growth for $\alpha>0$:}	
	\begin{itemize}
		\item {\bf Initial Rule} ($\sigma=0$):		
			To initialize the process, we start with $G_0$ being the empty graph.
		
		\item {\bf Growth Rule} (at time $\sigma+1$):		
			We choose vertex $x_{\sigma+1}$ uniformly at random in $S$ and add it to $G_\sigma$ with $m$ directed edges emanating from the vertex $x_{\sigma+1}$. 
			Let the endpoints of the $m$ emanating edges given by the vertices $\mv{\sigma+1}1,\ldots, \mv{\sigma+1}m$.
			We choose the endpoints at random (with replacement) from $V_\sigma$, such that $v \in V_\sigma$ is chosen with probability
			\begin{align}
			\label{eq:ch:1}				\msprob[\sigma]{\mv{\sigma+1}i=v}&=\frac{(d_{\sigma}(v)+\delta)F_n(|x_v-x_{\sigma+1}|)}{\max\{ T_{\sigma,n}(x_{\sigma+1}), \alpha\mexpec{T_{\sigma,n}(x_{\sigma+1})}/2\}},&
				\intertext{and}				\msprob[\sigma]{\mv{\sigma+1}i=\sigma+1}&=1-\frac{ T_{\sigma,n}(x_{\sigma+1})}{\max\{ T_{\sigma,n}(x_{\sigma+1}),\alpha\mexpec{T_{\sigma,n}(x_{\sigma+1})}/2\}},&
				\label{eq:ch:2}
			\end{align}
			for $i \in \{1,2,\ldots,m\}$.
	\end{itemize}
	
	\bigskip

	The above given random graph model is well defined, since the denominator is always strictly positive. Indeed, the following lemma calculates the value of $\mexpec{T_{\sigma,n}(x_{\sigma+1})}$ which is strictly positive. 
	\begin{lem}\label{m_dk(u)}
		For any fixed point $v\in S$, 
		\eqn{\label{calcIf}
		\int_S F_n(|v-u|)\,\mathrm{d}u=I_n,
		}
		where
		\eqn{ \nonumber
		I_n=\frac{1}{2}\int_{x=0}^\pi F_n(x)\sin x \,\mathrm{d} x.
		}
		As a consequence, if $U$ is a randomly chosen  point from $S$, then
		\eqn{\label{calcETn}
		\mexpec{T_{\sigma,n}(U)}=I_n(2m+\delta)\sigma.
		}
	\end{lem}
	\begin{proof}
		First note that $I_n$ does not depend on $v$ due to rotation invariance.		
		Thus, without loss of generality, we can assume that $v$ is at the north pole of the sphere. Using spherical coordinates, we find $du=r_0^2\sin \theta  \, \mathrm{d} \theta \,\mathrm{d} \varphi$, where $r_0=1/(2\sqrt{\pi})$, and $|v-u|=\theta$, so that:
		$$
		\int_SF_n(|v-u|)\,\mathrm{d}u=\int_{\varphi=0}^{2\pi}\int_{\theta=0}^{\pi}r_0^2 F_n(\theta) \sin \theta \,\mathrm{d} \theta \, \mathrm{d} \varphi =2\pi r_0^2\int\mproef{_{\theta=0}^\pi}{_0^\pi} F_n(\theta)\sin \theta\,\mathrm{d} \theta
		=I_n.
		$$
		For the second claim we calculate the expected value of $T_{\sigma,n}(U)$, \eqref{eq:Tt}, conditional on the graph $G_\sigma$:
		\begin{align*}
		\mexpec{T_{\sigma,n}(U)\,|\,G_\sigma}
		&=\sum_{v=1}^\sigma (d_\sigma(v) +\delta) \mexpec{F_n(|x_v-U|)\,|\,G_\sigma}
		\mproef{\\&}{\\&}=I_n\sum_{v=1}^\sigma (d_\sigma(v) +\delta) 
		=I_n (2m+\delta) \sigma,
		\end{align*}
		where we apply \eqref{calcIf} on $\mexpec{F_n(|x_v-U|)\,|\,G_\sigma}$, since $v \leq \sigma$ and, therefore, 
		\eqn{\label{eq.shrtcut.0}
		\mexpec{F_n(|x_v-U|)\,|\,G_\sigma}=\int_S F_n(|x_v-u|)d u=I_n.
		}
		Hence, $\mexpec{T_{\sigma,n}(U)}=\mexpec{\mexpec{T_{\sigma,n}(U)\,|\,G_\sigma}}=I_n (2m+\delta) \sigma$.
		\qed
	\end{proof}
	
	We use the abbreviations, for $u \in S$,
	\mproef{	\begin{align}
						\label{def:M}
						M_{\sigma,n}(u)&=\max\{ T_{\sigma,n}(u),\alpha \Theta I_n \sigma\}&
						&\text{and}&
						\mR(u)=F_n(|u-x_{\sigma+1}|),&
						\end{align}
				}{
						\begin{align}
						\label{def:M}
						M_{\sigma,n}(u)&=\max\{ T_{\sigma,n}(u),\alpha \Theta I_n \sigma\}&
						\mproef{&\text{and}&}{\nonumber\intertext{and}}
						\mR(u)&=F_n(|u-x_{\sigma+1}|),&
						\end{align}
				}
	where
	\eqn{\label{eq:def:theta}
		\Theta=\Theta(\delta,m)=(2m+\delta)/2.
		}
	
	As a consequence, we can rewrite the attachment rules as
	\begin{align}
\msprob[\sigma]{\mv{\sigma+1}i=v}&=\frac{(d_{\sigma}(v)+\delta)\mR(x_v)}{M_{\sigma,n}(x_{\sigma+1})}&
	\mproef{&\text{and}&}{\nonumber\intertext{and}}
	\msprob[\sigma]{\mv{\sigma+1}i=\sigma+1}&=1-\frac{ T_{\sigma,n}(x_{\sigma+1})}{M_{\sigma,n}(x_{\sigma+1})},&	
		\label{rule.sel}
	\end{align}
	for $i \in \{1,2,\ldots,m\}$.
	
	\begin{rem}
		In the above description we add directed edges to the graph and therefore we construct a directed graph.
		For questions about the connectedness and diameter of the graph we ignore the direction of the edges, but we need the direction of the edges in the proofs of the main results.
	\end{rem}

	\begin{rem}\label{canonical.functions}
	In this paper we will illustrate the theorems using the canonical functions
	\begin{align}
		F_n^\smallsup0(u)&\equiv1,& 
		F_n^\smallsup1(u)&=\indic{|u| \leq r_n} &
		&\text{and}&
		F_n^\smallsup2(u)&=\frac{1}{\max\{n^{-\psi},u\}^{\beta}},&
	\end{align}
	where $r_n \geq n^{\varepsilon-1/2}$, $\varepsilon < 1/2$, $\psi <1/2$ and $\beta \in (0,2) \cup (2,\infty)$.
	The canonical function $F_n^\smallsup0$ implies that the vertices are chosen proportional to the degree, and, furthermore, the geometry is ignored, the model is then equivalent to the PARID model, see \cite{parid?} or section \namedref{seq:heuristics}.
	The function $F_n^\smallsup1$ implies that a new vertex can only connect to vertices at distance at most $r_n$.
	Finally, canonical function $F_n^\smallsup2$ implies that vertices are chosen proportional to the degree, and, in contrast to $F_n^\smallsup0$, will prefer vertices close to the new vertex, since $F_n^\smallsup2$ is non-increasing as a function in $u$.
	\end{rem}

	\subsection{Heuristics and main results} \label{seq:heuristics}
	Using the results of \cite{flax07},
	which is a special case of our model when $\delta=0$, 
	together with the results of the \emph{PARID} model, 
	introduced in \cite{parid?}, 
	we will predict how the power-law exponent of the degree sequence will behave.
	
	Consider the PARID random graph process $\{G_\sigma'\}_{\sigma\geq0}$ as introduced in \cite{parid?} with constant weights equal to $m$.
	For this special case, we give a brief description of the model.
	
	The construction of the PARID graph $G_\sigma'=(V_\sigma',E_\sigma')$ depends on the graph $G_{\sigma-1}'$. 
	The rule of growth is as follow: add a vertex to the graph $G_{\sigma-1}'$ and from this vertex emanates $m$ edges. 
	The endpoints of these $m$ edges are chosen independently (with replacement) from the vertices of $G_{\sigma-1}'$. 
	The probability that vertex $v\in  V_{\sigma-1}'$ is chosen is proportional to the degree of vertex $v$ plus $\delta$, more specifically:
	\eqn{\label{parid.sel}
	\mprob{\text{choose vertex }v\,|\,G_{\sigma}'} =\frac{d_\sigma(v)+\delta}{(2m+\delta)\sigma}.
	}
	
	If $\alpha\leq 2$, $\delta>-m$ and $F_n =F_{n}^\smallsup0$, then the GPAF model coincides with the PARID model where the weight of each vertex is set to $m$. 	
	Note that, for the chosen parameters, 
	\begin{align*}
	A_{\sigma,n}(x_v)=F_n^\smallsup0(|x_v-x_{\sigma+1}|)&=1,& 
	I_n^\smallsup0&=\frac12\int_{x=0}^\pi \sin(x)\,\mathrm{d}x=1,&
	\end{align*}
	and
	\eqn{\nonumber
	\alpha \Theta I_n^\smallsup0 \sigma \leq 2\Theta \sigma=(2m+\delta)\sigma=\sum_{v\in V_\sigma} (d_\sigma(v) +\delta)=T^\smallsup0_{\sigma, n}(x_{\sigma+1}),
	}
	thus $M_{\sigma,n}^\smallsup0(x_{\sigma+1})=T_{\sigma,n}^\smallsup0(x_{\sigma+1}) =(2m+\delta)\sigma$.
	Therefore, the equations \eqref{rule.sel} turns into \eqref{parid.sel}, since
	$$
	\msprob[\sigma]{v_{\sigma+1}^{\smallsup{1}}=v}=
		\msprob[\sigma]{\text{choose vertex }v}
		=\frac{(d_\sigma(v)+\delta)A_{\sigma,n}(x_v)}{M^\smallsup0_{\sigma,n}(x_{\sigma+1})}
		=\frac{d_\sigma(v)+\delta}{(2m+\delta)\sigma}.
	$$
	Furthermore, note that for these parameters there are no self-loops, since
	$$
	\msprob[\sigma]{v_{\sigma+1}^{\smallsup{1}}=\sigma+1}=1-\frac{T^\smallsup0_{\sigma,n}(x_{\sigma+1})}{M^\smallsup0_{\sigma,n}(x_{\sigma+1})}=1-\frac{M^\smallsup0_{\sigma,n}(x_{\sigma+1})}{M^\smallsup0_{\sigma,n}(x_{\sigma+1})}=0.
	$$

	For the PARID model, we know that the power-law exponent is $3+\delta/m$, thus we expect that the power-law exponent in our model is $3+\delta/m$ if $\alpha \leq 2$ and $F_n=F_n^\smallsup0$. For $\alpha>2$, $\delta=0$ and $F_n$ satisfying some mild condition, see \eqref{F:condition}, we know from \cite{flax07} that the power-law exponent is $1+\alpha$, which is independent of $F_n$.
	
	We will show in this paper that the power-law exponent is given by $1+\alpha(1+\delta/2m)$, which generalizes the two mentioned papers \cite{flax04, flax07}.
	More precisely, let $N_k(\sigma)$ denote the number of vertices of degree $k$ in $G_\sigma$ and let $\bar N_k(\sigma)$ be its expectation. We will show that:	
	\begin{thm}[Behavior of the degree sequence] \label{theorem:1}
		Suppose that $\alpha>2$, $\delta>-m=m(n)$ and in addition that for $n\rightarrow \infty$,
		\eqn{\label{F:condition}
		\int_{x=0}^\pi F_n(x)^2 \sin x \,\mathrm{d}x = \mO{n^{\theta}I_n^2},
		}
		where $\theta<1$ is a constant. Then there exists a constant $\gamma_1>0$ such that for all $k=k(n) \geq m$,
		\eqn{\label{thm:claim1}
		\bar N_k(n)=n e^{\phi_k(m,\alpha,\delta)}\br{\frac{m}{k}}^{1+\alpha(1+\delta/2m)} +{\cal O}(n^{1-\gamma_1}),
		}
		where $\phi_k(m,\alpha,\delta)$ tends to a constant $\phi_\infty(m,\alpha,\delta)$ as $k\rightarrow \infty$.
		
		Furthermore, for each $\epsilon>0$ and $n$ sufficiently large, the random variables $N_k(n)$ satisfy the following concentration inequality
		\eqn{\label{thm:claim2}
		\mprob{|N_k(n)-\bar N_k(n)| \geq I_n^2n^{\max\bra{1/2,2/\alpha}+\epsilon}} \leq e^{-n^\epsilon}.
		}	
	\end{thm}	
			\begin{rem}
			Note that the power-law exponent in \eqref{thm:claim1} does not depend on the choice of the function $F_n$. We will see in the proof, that $F_n$ manifests itself only in the error terms.
			
			All the given canonical functions in Remark \ref{canonical.functions}	do satisfy the condition given by \eqref{F:condition}: it should be evident that for $F_n^\smallsup0$ the constants $I_n$ and $\theta$ are given by $I_n^\smallsup0=1$ and $\theta^\smallsup0=0$, respectively. Furthermore, in  \cite{flax07} it is shown that one can take $I_n^\smallsup1\sim r_n^2/4$, $I_n^\smallsup2 =\mO{1}$ if $\beta \in (0,2)$, and  $I_n^\smallsup2\sim \frac{n^{\delta(\beta-2)}}{2(\beta-2)}$ if $\beta>2$, and, hence we can take  $\theta^\smallsup1=0$, $\theta^\smallsup2=0$ and $\theta^\smallsup2=2\psi$, respectively.
	\end{rem}

Before we consider the connectivity and diameter of $G_n$, we place some additional restrictions on the function $F_n$. 
These restrictions are necessary to end up with a graph which is with high probability connected. Keep in mind the function $F^\smallsup1_n(u)=\indic{|u|\leq r_n}$, then it should be clear that $r_n$ should not decrease too fast, otherwise we end up with a disconnected graph.
	
	Let $\rho_n=\rho(\mu,n)$ be such that
	$$
	\mu I_n=\frac12\int_{x=0}^{\rho_n} F_n(x)\sin x\,\mathrm{d}x,
	$$
	for some $\mu \in(0,1]$.
	
	We will call $F_n$ \emph{smooth} (for some value of $\mu$) if
	\begin{itemize} \label{assumptions}
		\item[]{\bf (S1)} $F_n$ is monotone non-increasing;
		\item[]{\bf (S2)} $n\rho_n^2 \geq L \log n$, for some sufficiently large $L$;
		\item[]{\bf (S3)} $\rho_n^2F_n(2\rho_n) \geq c_3 I_n$, for some constant $c_3$ which is bounded from below.
	\end{itemize}
	
	Before stating the theorem, we will give an intuitive meaning of $\rho_n$.
	To that end, consider the function $F^\smallsup1_n(u)=\indic{u<r_n}$ and  use the fact that if $r_n =\mO{n^{-1/2-\varepsilon}}$ for some $\varepsilon>0$, then the limiting graph is not connected, see \cite{Penrose}.
	It should be intuitively clear that in the limit, each newly added vertex $x_n$ should connect to at least one other vertex. Thus, there should be at least one vertex within distance $r_n$ of $x_n$.
	 At time $n$ there are $n-1$ vertices and the probability  that at least one of these vertices is at distance at most $r_n$ of vertex $x_{n}$, denoted by $p_{\mathrm c}(n,r_n)$, is at most $C(n-1)r_n^2$ for some constant $C$. On the other hand, we see that if $r_n =\mO{n^{-1/2-\varepsilon}}$ then $p_{\mathrm c}(n,r_n)=\mO{n^{-2\varepsilon}}$ tends zero for large $n$, and, as a consequence, in the limit the graph is not connected. 
	 If, as is our assumption, $r_n>n^{\varepsilon-1/2}$, then $p_{\mathrm c}(n, r_n)\rightarrow \infty$. 
	 
	 Interpret $\rho_n$ for general $F_n$ as the radius. 	 
	 The condition that $p_{\mathrm c}(n, r_n) \rightarrow \infty$ is replaced by
	 $n\rho_n^2 > L\log n$. Then, intuitively,  condition {\bf S2} implies that for general $F_n$ the value $p_{\mathrm c}(n,\rho_n)$, does not tend to zero and implies that the limiting graph is connected. 		 
	 The conditions {\bf S1} and {\bf S3} are technicalities, combined they ensure that the `area' due to the radius $\rho_n$ is sufficiently large: condition {\bf S1} states that $F_n$ is monotone non-increasing and combined with {\bf S3} one can show that the `area' within radius $2\rho_n$ is $(2\rho_n)^2F_n(2\rho_n)$, which is at least $4c_3 I_n$.

	\begin{thm} \label{thm:diameter}
	If $\alpha \geq 2$ and $F_n$ is smooth, $m \geq K\log n$, and $K$ is sufficiently large constant,
		then with high probability
		\begin{itemize}
			\item $G_n$ is connected;
			\item $G_n$ has diameter $\mO{\log n/\rho_n}$.
		\end{itemize}
	\end{thm}

	\begin{rem}
		All the canonical functions are smooth. 
		It should be evident that one can take for $F^\smallsup0_n$: $\mu^\smallsup0=1$, $\rho_n^\smallsup0=1$ and $c_3^\smallsup0=1$ for $F^\smallsup0_n$. 	
		For $F^\smallsup1_n(u)$ one can take for example $\mu^\smallsup1\sim1/4$ and $\rho_n^\smallsup1=r_n/2$ and $c_3^\smallsup1 \sim 1$. 
				Finally, $F_n^\smallsup2$  is also smooth, we refer to \cite{flax07} for the precise values of $\rho_n^\smallsup2,\mu^\smallsup2$ and $c_3^\smallsup2$. 
	\end{rem}

	We end with a sharper result on the diameter, however we, also, need  stronger restrictions on the function $F_n$. We will call $F_n$ \emph{tame} if there exists strictly positive constants $C_1$ and $C_2$ such that
	\begin{itemize}
		\item[]{\bf (T1)} $F_n(x)\geq C_1$ for $0 \leq x \leq \pi$;
		\item[]{\bf (T2)} $I_n \leq C_2$.
	\end{itemize}

	\begin{thm} \label{thm:diameter2}
	If $\alpha \geq 2$, $\delta > -m$ and $F_n$ is tame and $m \geq K\log n$, and $K$ sufficiently large,
		then with high probability
		\begin{itemize}
			\item $G_n$ is connected;
			\item $G_n$ has diameter $\mO{\log_m n}$.
		\end{itemize}
	\end{thm}

	\begin{rem}
		It should be evident that the function $F_n^\smallsup0$ is tame, since one can take $C_1=C_2=1$. If $\beta \in(0,2)$ then we also have that $F_n^\smallsup2$ is tame, since
		\begin{align*}
		F_n^\smallsup2(x) &\geq \pi^{-\beta}, \text{ for $0\leq x \leq \pi$,}&
		&\text{and}&
		I_n^\smallsup2&=\frac12\int_{x=0}^\pi x^{-\beta}\sin x \,\mathrm{d}x \leq \frac{\pi^{2-\beta}}{2(2-\beta)}.&
		\end{align*}
	\end{rem}

	\begin{rem}
	If we consider the configuration model (CM), see \namedref{sec:related}, \cite{finstub, infvar}  or the Poissonian random graph (PRG), see \cite{EVshort, norros3}, then the typical distance depends on the power-law exponent. If the power-law exponent is larger than $3$, then the typical distance is of $\mO{\log n}$, where $n$ is the number of vertices in the graph, and if the power-law exponent is between 2 and 3 then the typical distance is of $\mO{\log \log n}$.
	On before hand, it is not clear if this holds for the GPAF model. Theorem \ref{thm:diameter}, only states an upper bound on the diameter, independent of $\delta$.
	
	If $F_n=F_n^\smallsup0\equiv 1$ and $\delta \in(-m,0)$ then the authors of \cite{diaminpref} show that the diameter in the graph $G_n$ fluctuates around $\log\log n$. If $F_n(u)=F_n^\smallsup1(u)=\indic{|u| \leq r_n}$, then, intuitively, the diameter depends only on $r_n$, since $r_n$ determines the maximal length of an edge, and we  conjecture that the diameter is at least of order $\log n$.
	\end{rem}

		\subsection{Related work} \label{sec:related}
		In this section we consider random graph models,
which are related to the Geometric Preferential Attachment model with fitness (GPAF).

As mentioned earlier the model is related to the Albert-Bar\'abasi (BA) model.
In the BA-model the power-law exponent $\tau$ is
limited to the value $3$, which was proven by
Bollob\'as and Riordan.

Cooper and Frieze introduced in \cite{ccooafrie01}
a very general model preferential attachment model.
In this model it is both possible to introduce
new vertices at each time step or to introduce
new edges between old vertices.
Due to the weights with which edges of the new vertices are attached
to old vertices and the adding of
edges between old vertices, the power-law exponent $\tau$ can obtain any value $\tau>2$.

In \cite{PhysRevLett.85.4633} the authors overcome the restriction $\tau \geq 3$
in a different way, by choosing the endpoint of an edge proportional
to the in-degree of a vertex plus some initial attraction $A > 0$.
This is identical by choosing the endpoint of an edge proportional
to the degree of a vertex plus some amount $\delta=A-m > -m$,
as done in the PARID model (cf. \cite{parid?}).
The power-law exponent in \cite{PhysRevLett.85.4633} is given by $\tau=3+\delta/m$.
Note that for $\delta=0$ we obtain the BA model. The authors of \cite{parid?}
show more rigorously some of the results in \cite{PhysRevLett.85.4633}.

Both in \cite{ccooafrie01}
and in \cite{parid?} it is allowed to add  a random number of edges $W$,
with the introduction of a new  vertex.
In case the mean of $W$ is finite the power-law exponent is given by
$\tau=3+\delta/\mexpec{W}$. Hence, if $\tprob(W= m)=1$ for some integer $m\geq 1$
then we see that $\tau =2+\delta/m \geq 2$, since we can choose for
$\delta $ any value in $(-m, 0)$.

In \cite{flax04, flax07} the authors add geometry to the BA model,
which corresponds to the GPAF model, introduced above, with $\delta=0$.
Due to a technical difficulty the model has an additional parameter, called $\alpha > 2$.
As a consequence of this restriction they only obtain power-law exponents
greater than $3$, since the power-law exponent is given by $\tau=\alpha+2$.

By combining the GPA and PARID model, we obtain the GPAF model,
introduced in this paper.
Due to the additional parameter $\delta$,
it is in this model possible to obtain any power-law exponent $\tau$ bigger than $2$.

	\subsection{Overview of the paper}
		The remainder of this paper is divided into three sections. 		
		In \namedref{sec:recurrence} we will derive a recurrence relation for the expected number of vertices of a given degree. In \namedref{sec:Coupling} we will present a coupling between the graph process and an urn scheme, which will be used in \namedref{sec:proof.main} to show that the number of vertices with a given  degree is concentrated around its mean. 
		
		\section{Recurrence relation for the expected degree sequence} \label{sec:recurrence}
		In this section we will establish a recurrence relation for $\bar N_k(\sigma)=\mexpec{N_k(\sigma)}$, the expected number of vertices with degree $k$ at time $\sigma$, which is claim \eqref{thm:claim1} of Theorem \ref{theorem:1}. 
		From this recurrence relation, we will show that
		$$
		\bar N_k(\sigma) \sim \sigma p_k,
		$$
		where $p_k \sim k^{-(1+\alpha(1+\delta/2m))}$, as $k \rightarrow \infty$.
		The proof of claim \eqref{thm:claim1} depends on a lemma, which is crucial for the proof. This lemma states that for sufficiently large $n$ the value  $M_{\sigma,n}(x_{\sigma+1})$ is equal to $\alpha \Theta  I_n\sigma$, with high probability. This is a consequence of the fact that $T_{\sigma,n}(x_{\sigma+1})$ is concentrated around its mean $\mexpec{T_{\sigma,n}(x_{\sigma+1})}=2\Theta I_n \sigma < \alpha \Theta I_n \sigma$, see \eqref{calcETn} and \eqref{eq:def:theta}, which is the content of the next lemma.
			\begin{lem} \label{conc.W}
			If $\alpha>2$, $\delta>-m$, $\sigma=0,1,2,\ldots,n$, and $U$ is chosen randomly from $S$ then
			\eqn{\nonumber
			\mprob{\brc||{T_{\sigma,n}(U)-\mexpec{T_{\sigma,n}(U)}} > \Theta I_n \br{\sigma^{2/\alpha}+\sigma^{1/2}\log \sigma}\log n}=\mO{n^{-2}}.
			}
		\end{lem}	
		The proof of this lemma is deferred to \namedref{subs conc.W}.
 	
 	We will allow that $m$ depends on $n$, thus $m=m(n)$, as already pointed out previously. 
		In establishing the recurrence relation for $\bar N_{k}(\sigma)$, we will rely on the derivation for $\delta=0$ in \cite[Section 3.1]{flax07}.
		
At each time, we add a new vertex from which $m$ edges are emanating, and  for each of these $m$ emanating edges we need to choose a vertex-endpoint. 
	The first possibility for a vertex to have degree $k$ at time $\sigma +1$ is that the degree at time $\sigma$ was equal to $k$ and that none of the $m$ endpoints, emanating from $x_{\sigma+1}$, attaches to the vertex.
	Furthermore, ignoring for the moment the effect of selecting the same vertex twice or more, the vertex could also have degree $k-1$ at time $\sigma$ and having one endpoint attached to it at time $\sigma+1$. 
	Finally, it is also possible that the newly added vertex $x_{\sigma+1}$ has degree $k$.	
The total number of vertex-endpoints with degree $k$ is distributed as $\mathrm{Bin}\br{m, p_k(\sigma)}$, where
\eqn{\label{eq:p_k}
p_k(\sigma)= \sum_{v\in  D_k(\sigma)}\frac{(k+\delta)\mR(x_v)}{M_{\sigma,n}(x_{\sigma+1})},
}
and $D_k(\sigma) \subset V_\sigma$ is the set of vertices with degree $k$ in the graph $G_\sigma$.
Similarly, the number of vertex-endpoints with degree $k-1$ is distributed as \mproef{$}{$$}\mathrm{Bin}\br{m, p_{k-1}(\sigma)}.\mproef{$}{$$} If the newly added vertex $x_{\sigma+1}$ ends up with degree $k$, then this vertex has $k-m$ self-loops. The number of self-loops, $d_{\sigma+1}(\sigma+1)-m$, is distributed as $\mathrm{Bin}\br{m,p}$, where
\eqn{\label{distbin}		
p=1-{T_{\sigma,n}(x_{\sigma+1})}/{M_{\sigma,n}(x_{\sigma+1})}.
}
For $k\geq m$, this leads  to, 
		\begin{align}\nonumber
		\mexpec{N_k(\sigma+1) | G_\sigma,x_{\sigma+1}}&=
		N_k(\sigma)
				-mp_k(\sigma)
				+mp_{k-1}(\sigma)
		\\&\,\,\,\,\,\,\,\,+\mexpec{\indic{d_{\sigma+1}(\sigma+1)=k}\,|\,G_\sigma,x_{\sigma+1}}	+\mO{m\eta_k(G_\sigma,x_{\sigma+1})},\label{rec.goal}
		\end{align}
		where $\eta_k(G_\sigma,x_{\sigma+1})$ denotes the probability, conditionally on $G_\sigma$, that the same vertex-endpoint is chosen at least twice and at most $k$ times.

	Taking expectations on both sides of \eqref{rec.goal}, we obtain
			\begin{align}
		\bar N_k(\sigma+1)&=\nonumber
		\bar N_k(\sigma)				
				-m\mexpec{\sum_{v\in D_k(\sigma)} {\frac{(k+\delta)\mR(x_v)}{M_{\sigma,n}(x_{\sigma+1})}}}
				\\&\nonumber \quad +m\mexpec{\sum_{v\in D_{k-1}(\sigma)}\frac{(k-1+\delta)\mR(x_v)}{M_{\sigma,n}(x_{\sigma+1})}}
		\\&\quad+\mprob{d_{\sigma+1}(\sigma+1)-m=k-m}	+\mO{m\mexpec{\eta_k(G_\sigma,x_{\sigma+1})}}. \label{rec.d.goal}
		\end{align}

	Let 
		\eqn{ \label{eq:CA}
		{\cal B}_\sigma=\bra{\brc||{T_{\sigma,n}(x_{\sigma+1})-2\Theta I_n \sigma} < C_1\Theta I_n  \sigma^{\gamma}\log n},
		}
	where $\max\{2/\alpha, \theta\} < \gamma < 1$ and $C_1$ is some sufficiently large constant. 
	If 
	$$\sigma \geq t_0=t_0(n)=(\log n)^{2/(1-\gamma)},
	$$
	then ${\cal B}_\sigma$ implies that for sufficiently large $n$, 
	$$
	T_{\sigma,n}(x_{\sigma+1}) \leq 2\Theta I_n \sigma + C_1 \Theta I_n  \sigma^{\gamma} \log n 
	= 2\Theta I_n \sigma \br{1+\mO{\log^{-1} n}} 
	\leq \alpha \Theta I_n \sigma,
	$$
	since $\alpha>2$, and, hence, with high probability
	$$
	M_{\sigma,n}(x_{\sigma+1})=\max\{T_{\sigma,n}(x_{\sigma+1}), \alpha \Theta I_n \sigma\} = \alpha \Theta I_n \sigma.
	$$
	
	Next, we consider each term on the right hand side of \eqref{rec.d.goal} separately, for $\sigma=1,2,\ldots,n$.	
	For the first two terms on the right hand side of \eqref{rec.d.goal} we will use that $p_k(\sigma)$ is a probability and that $\mprob{{\cal B}_n^c} =\mO{n^{-2}}$, for $\sigma >t_0$, see Lemma \ref{conc.W}, which yields
	\eqn{ \label{rel:x:cond}
	\mexpec{p_k(\sigma)}=\mexpec{p_k(\sigma)\,|\, {\cal B}_n}\mprob{{\cal B}_n}+\mexpec{p_k(\sigma)\,|\, {\cal B}_n^c}\mprob{{\cal B}_n^c}
	=\mexpec{p_k(\sigma)\,|\, {\cal B}_n}+\mO{n^{-2}}.
	}
	Also, using that $N_k(\sigma) \leq \sigma$,
	\eqn{ \label{rel:x:cond2}
	\mexpec{ N_k(\sigma) \,|\, {\cal B}_n} 
		= \frac{\mexpec{N_k(\sigma)}-\mexpec{N_k(\sigma)\,|\,{{\cal B}_n^c}} }{\mprob{{\cal B}_n}}
		=\bar N_k(\sigma)+\mO{\sigma n^{-2}}.
	}
	For $\sigma$ sufficiently large, using \eqref{calcIf}, \eqref{def:M}, \eqref{eq:p_k} and \eqref{rel:x:cond2},
	\begin{align} \nonumber
	\mexpec{p_k(\sigma)\,|\, {\cal B}_n}&=\frac{k+\delta}{\alpha\Theta I_n \sigma}\mexpec{\brc.|{\sum_{v\in D_{k}(\sigma)}\mR(x_v)\,}\, {\cal B}_n}	
	\nonumber
	\\&\nonumber=\frac{k+\delta}{\alpha\Theta I_n \sigma}\mexpec{\brc.|{\sum_{v\in D_{k}(\sigma)}
	\int_S F_n(|x_{v}-u|)\,\mathrm d u \,}\, {\cal B}_n}	
	\\&=\frac{(k+\delta)\mexpec{N_k(\sigma) \,|\, {\cal B}_n}}{\alpha\Theta \sigma}
	=\frac{(k+\delta)\bar N_k(\sigma)}{\alpha\Theta \sigma}+\mO{kn^{-2}} \label{impact_n}.
	\end{align}	
	Combining \eqref{rel:x:cond}  and \eqref{impact_n}, we obtain
		\eqn{ \label{master.eq}	
				\mexpec{p_k(\sigma)}=\mexpec{\sum_{v\in D_{k}(\sigma)}{\frac{(k+\delta)\mR(x_v)}{M_{\sigma,n}(x_{\sigma+1})}}}
					=\frac{(k+\delta){\bar N_k(\sigma)}}{\alpha \Theta \sigma}+\mO{kn^{-2}},
		}
	for $\sigma \geq t_0=t_0(n)=(\log n)^{2/(1-\gamma)}$. The above statement remains true when we replace $k$ by $k-1$.

	For the third term on the right hand of \eqref{rec.d.goal}, one can show that for $\sigma \geq t_0$, using that $d_{\sigma+1}(x_{\sigma+1})-m$ has a binomial distribution, see \eqref{distbin}, thus
	\begin{multline*}	
		\mprob{d_{\sigma+1}(x_{\sigma+1})-m=k-m\,|\,{\cal B}_n}
			=\binom{\!m\!}{\!k-m\!}\mexpec{\brc.|{p^{k-m}(1-p)^{2m-k}\,}\,{\cal B}_n}
			\\=\binom{\!m\!}{\!k-m\!}\br{1-\frac2\alpha}^{k-m}\br{\frac2\alpha}^{2k-m}\br{1+\mO{\sigma^{\gamma-1}\log n}},
	\end{multline*}
	where we refer to  \cite[\S3.1]{flax07} for the derivation of the above result. It follows that
	\begin{multline} \label{bin.fin}
	\mprob{d_{\sigma+1}(x_{\sigma+1})=k}
	= \mprob{d_{\sigma+1}(x_{\sigma+1})-m=k-m\,|\,{\cal B}_n}\mprob{{\cal B}_n}+\mO{\mprob{{\cal B}_n^c}}\\=\binom{\!m\!}{\!k-m\!}\br{1-\frac2\alpha}^{k-m}\br{\frac2\alpha}^{2k-m}+\mO{\sigma^{\gamma-1}\log n},
	\end{multline}
	where we refer to \cite{flax07} for the derivation of the error term $\mO{\sigma^{\gamma-1}\log n}$.
	
	For the fourth and final term on the right hand side of \eqref{rec.d.goal}, we use
	$$
	\eta_k(G_\sigma,x_{\sigma+1})
	=\mO{\min\bra{\sum_{i=m}^k\sum_{v\in D_i(\sigma)} \frac{(i+|\delta|)^2\mR(x_v)^2}{M_{\sigma,n}(x_{\sigma+1})^2},1}},
	$$
	which generalizes Equation (5) in \cite{flax07}.
	Using similar arguments that led to \eqref{master.eq}, 
	one can show for 	
	\begin{align} \label{vorw}
		\sigma > t_1  &=t_1(n)=n^{(\gamma+\theta)/2\gamma}& 
		&\text{and}&
		k \leq k_0=k_0(n)=n^{(\gamma-\theta)/4},
	\end{align}
	that 
	\eqn{ \label{par.afs}
	\mexpec{m\eta_k(G_\sigma,x_{\sigma+1})}
		=\mO{\frac{k^2n^\theta}{m\sigma}}=\mO{\sigma^{\gamma-1}}.
	}
		
	Substituting \eqref{master.eq}, \eqref{bin.fin} and \eqref{par.afs} in \eqref{rec.d.goal}, we end up with the following recurrence relation:
		\begin{multline} \label{fault.master}
		\bar N_k(\sigma+1)=\bar N_k(\sigma)  -\frac{m}{\alpha \Theta}(k+\delta)N_k(\sigma)/\sigma+ \frac{m}{\alpha \Theta}(k-1+\delta) \bar N_{k-1}(\sigma)/\sigma
			\\+\indic{m \leq k \leq 2m}\binom{\!m\!}{\!k-m\!}\br{1-2\alpha^{-1}}^{k-m}\br{2\alpha^{-1}}^{2k-m}+\mO{\sigma^{\gamma-1}\log(n)},
		\end{multline}
	for $k \geq m$ and $\bar N_{m-1}(\sigma)=0$ for all $\sigma \geq 0$.
	The above recurrence relation depends on $\sigma$ and $k$.	
	Consider the limiting case, i.e., $\sigma \rightarrow \infty$, and assume that for each $k$ the limit
	\eqn{\label{eq:assump}
	\bar N_k(\sigma)/\sigma \rightarrow p_k
	}	
	exists.
	If this is indeed the case, then in the limit the recurrence relation \eqref{fault.master} yields:
		\begin{multline*}
		p_k=\frac{m}{\alpha \Theta}(k-1+\delta)p_{k-1}-\frac{m}{\alpha \Theta}(k+\delta)p_k \mproef{\\}{\\}+\indic{m \leq k \leq 2m}\binom{\!m\!}{\!k-m\!}\br{1-2\alpha^{-1}}^{k-m}\br{2\alpha^{-1}}^{2k-m},
		\end{multline*}	
		where $k\geq m$ and $p_{m-1}=0$.
	By induction, we then obtain, for $k > 2m$,
		$$
		p_k
					=\frac{\frac{m}{\alpha \Theta}(k-1+\delta)}{1+\frac{m}{\alpha \Theta}(k+\delta)}p_{k-1}
				=\frac{k-1+\delta}{k+\delta+\frac{\alpha \Theta}{m}}p_{k-1}
					=\frac{\Gamma(m+1+\delta+\frac{\alpha \Theta}{m})\Gamma(k+\delta)}{\Gamma(m+\delta)\Gamma(k+1+\delta+\frac{\alpha \Theta}{m})}p_{2m}.
		$$
		Using that $\Gamma(t+a)/\Gamma(t) \sim t^a$ for $a\in[0,1)$ and $t$ large, we can rewrite the above equation  as follow:
		$$
		p_k=\phi_k(m,\alpha,\delta)\br{\frac{m}{k}}^{1+\frac{m}{\alpha \Theta}}=\phi_k(m,\alpha,\delta)\br{\frac{m}{k}}^{1+\alpha(1+\delta/2m)},
		$$
		where $\phi_k(m,\alpha,\delta)=\mO{1}$ and tends to the limit $\phi_\infty (m, \alpha, \delta)$ depending only on $m,\alpha$ and $\delta$ as $k \to \infty$.
	Finally, following the proof in \cite[from equation (15) up to the end of the proof]{flax07}, which shows that there exists a constant $M$ independent from $n$, such that
	\eqn{ \label{eq:inlsuit}
	|\bar N_ k (\sigma) - p_k \sigma|\leq M (n^{1-(\gamma-\theta)/4}+\sigma^\gamma\log n),
	}
	for all $0\leq \sigma \leq n$ and $m \leq k \leq k_0(n)$. 
	Thus, the assumption \eqref{eq:assump} is satisfied.
	By picking $\gamma_1>0$ sufficiently small, we can replace the right hand of \eqref{eq:inlsuit} by $n^{1-\gamma_1}$, and one obtains the claim \eqref{thm:claim1}.

	\section{Coupling}	\label{sec:Coupling}	
		In this section we make preparations for the proofs of Lemma \ref{conc.W} and the concentration result in Theorem \ref{theorem:1}, see \eqref{thm:claim2}. In this section we take $\tau\in \{1,\ldots,n\}$ fixed and we consider the graph process up to time $\tau-1$ resulting in the graph $G_{\tau-1}$. At time $\tau$ we apply the {\bf Growth Rule}, see \namedref{gpaf:sec:model}, twice on $G_{\tau-1}$, independently of each other, which results in the graphs  $G_\tau$ and $\hat G_\tau$.
		The idea is to compare the graphs $G_\tau$ and $\hat G_\tau$ over time by considering $G_\sigma$ and $\hat G_\sigma$ for $\tau \leq \sigma \leq n$. 
	 	To this end, we will introduce two urn processes. 
	 	The urns consist of weighted and numbered balls.
		Instead of choosing a vertex-endpoint $v\in V_{\sigma+1}$ at time $\sigma+1$ by \eqref{eq:ch:1} and \eqref{eq:ch:2}, we will draw (with replacement) a ball proportional to its weight  and then the vertex-endpoint is given by the number on the ball.

				The coupling between the urns will be introduced in four steps. 
		The first step is to introduce for any $\sigma \geq \tau$ two urns.
		Secondly, we will introduce a probabilistic coupling between the two urn processes.		
		Thirdly, we will describe the coupling between the graph processes $G_\sigma$, $\hat G_\sigma$ and the two urn processes.
		Finally, we consider the vertex-endpoints $\mv{\sigma}i$ and $\mhv{\sigma}i$, for $i=1,2,\ldots,m$, in the graphs $G_\sigma$ and $\hat G_\sigma$, respectively, and we will calculate the probability that $\mv{\sigma}i\not=\mhv{\sigma}i$.
	
	\subsection{The two urns} \label{seq:the urns}					
		In this section we describe the contents of the urns corresponding to the graphs $G_\sigma$ and $\hat G_\sigma$, for $\sigma=\tau,\tau+1,\ldots,n,$ and we give an alternative way of choosing the vertex-endpoints using the urns.

		Fix two graph processes $\{G_s\}$ and $\{\hat G_s\}$ such that the graphs up to time $\tau-1$ are identical, i.e., $G_s=\hat G_s$ for $s=0,1,2\ldots,\tau-1$, and that $x_s=\hat x_s$, for $s=\tau+1,\tau+2, \ldots, n$. Thus, the points $x_\tau$ and $\hat x_\tau$ will differ from each other, and, as a consequence, also, the edge sets $E_s$ and $\hat E_s$, for $s=\tau,\tau+1,\ldots,\sigma$, will be different.  	
		Finally, we assume, without loss of generality, that $T_{\sigma,n}(x_{\sigma+1}) \leq \hat T_{\sigma,n}(x_{\sigma+1})$.  	
		
		Next, we will describe the contents of the urns $U_\sigma$ and $\hat U_\sigma$ given the graphs $G_\sigma$ and $\hat G_\sigma$, and the newly added vertex $x_{\sigma +1}$. We will use the following abbreviations:
		\begin{align}
		T_{\sigma,n}&=T_{\sigma,n}(x_{\sigma+1})&
		&\text{and}&
		M_{\sigma,n}&=M_{\sigma,n}(x_{\sigma+1}).&				
		\end{align}		
		Furthermore, if $e$ is an edge, then we denote by $\head{e}$ the endpoint of the edge. Thus, if edge $e$ is added at time $t$, emanating from the vertex $t$, points to a vertex $s \in V_t$ then $\head{e}=s$.

		\bigskip\noindent {\bf Contents of the urns:} 
		\begin{itemize}
			\item For each edge $e\in E_\sigma$, such that $\head{e}\not =\tau$, there is a white ball in $U_\sigma$ of weight $\mR(x_{\head{e}})$ and numbered $\head{e}$. Similarly, for each edge in $e\in \hat E_\sigma$, such that $\head{e}\not =\tau$, there is a white ball in $\hat U_\sigma$ of weight $\mR(\hat x_{\head{e}})=\mR(x_{\head{e}})$ and numbered $\head{e}$.  Observe that $\hat x_{\head{e}}=x_{\head{e}}$ since $\head{e}\not =\tau$.

			\item For each vertex $v\in V_\sigma \backslash \{\tau\}$ there is a red ball in each of the urns $U_\sigma$ and $\hat U_\sigma$ of weight $(m+\delta)\mR(x_v)$ and numbered $v$.
			
			\item For the vertex $\tau$ there is in $U_\sigma$ a purple ball of weight $(d_\sigma(\tau)+\delta)\mR(x_\tau)$ and number $\tau$, and in $\hat U_\sigma$ there is an orange ball of weight $(\mhX d(\tau)+\delta)\mR(\hat x_\tau)$ and numbered $\tau$.
			
			\item For the vertex $\sigma + 1$ each of the urns $U_\sigma$ and $\hat U_\sigma$ contain a green ball of weight $(\alpha\Theta I_n \sigma-\hat T_{\sigma,n})^+$, where $(\cdot)^+=\max\{0, \cdot \}$, and numbered $\sigma+1$. Furthermore, we add only to $U_\sigma$ a blue ball of weight $((\alpha\Theta I_n \sigma-T_{\sigma,n})^+-(\alpha\Theta I_n \sigma-\hat T_{\sigma,n})^+)^+$ and numbered $\sigma+1$.
		\end{itemize}
		\bigskip
		\begin{rem}\label{rem:total_w}
		The total weight of the white and red balls in $U_\sigma$ are given by
		\begin{align*}
		&\sum_{e\in E_\sigma}A_{\sigma,n}(x_{\head{e}})\indic{\head{e}\not = \tau}&
		&\text{and}&
		&\sum_{v \in V_\sigma \backslash\{\tau\}}(m+\delta)A_{\sigma,n}(x_v),&		
		\end{align*}
		respectively, and the weight of the purple ball in $U_\sigma$ can be rewritten as
		$$
		(d_\sigma(\tau)+\delta)A_{\sigma,n}(x_\tau)=\sum_{e\in E_\sigma}A_{\sigma,n}(x_{\head{e}})\indic{\head{e}= \tau}+(m+\delta)A_{\sigma,n}(x_\tau).
		$$
		Therefore, the total weight of the white, red and purple balls in $U_\sigma$ is equal to:
		$$
		\sum_{e\in E_\sigma}A_{\sigma,n}(x_{\head{e}})
		+
		\sum_{v \in V_\sigma}(m+\delta)A_{\sigma,n}(x_v)
		=\sum_{v \in V_\sigma}(d_\sigma(v)+\delta)A_{\sigma,n}(x_v)=T_{\sigma,n}.
		$$		
		Furthermore, from \eqref{def:M}, and some easy calculation, the total weight of all the balls  in $U_\sigma$ is $M_{\sigma,n}$.
	Similarly, the total weight of the white, red and orange balls in the urn $\hat U_\sigma$ is $\hat T_{\sigma,n}$ and the total weight of all the balls in $\hat U_\sigma$ is, precisely $\hat M_{\sigma,n}$.
		\end{rem}
			The weight of a ball depends on the time $\sigma$, the color of the ball and the number on the ball. Let $b$ be a ball in $U_\sigma$ or $\hat U_\sigma$, then we define the weight function $w_\sigma$ as
		\newcommand{\mxi}{\xi}
		\begin{align}\label{eq:weights}
		w_\sigma(b)= 
			\begin{cases}			
			\mR(x_{\mxi(b)}) & \text{if } b \text{ is white},\\
			(m+\delta)\mR(x_{\mxi(b)}) & \text{if } b \text{ is red},\\			
			(d_\sigma(x_\tau)+\delta)\mR(x_\tau)& \text{if } b \text{ is purple},\\
			(\hat d_\sigma(\hat x_\tau)+\delta)\mR(\hat x_\tau)& \text{if } b \text{ is orange},\\
			(\alpha\Theta I_n \sigma-\hat T_{\sigma,n})^+ & \text{if } b \text{ is green},\\
			((\alpha\Theta I_n \sigma-T_{\sigma,n})^+-(\alpha\Theta I_n \sigma-\hat T_{\sigma,n})^+)^+ & \text{if } b \text{ is blue},												
			\end{cases}		
		\end{align}
		where $\xi(b)$ is the number on the ball.
		Observe that the number and the color together determine the weight of a ball.

	We identify a set $B \subset U_\sigma$ or $B \subset \hat U_\sigma$ of distinct balls by the set of pairs $(c,k)$, where $c$ denotes the color and $k$ the number of the ball.
		For any set $B$ of distinct balls, define
	$$
	\mn{B}=\sum_{b\in B}w_\sigma(b).
	$$		
	
	We will draw the balls $\{\mb\sigma i\}_{i=1}^m$ with replacement from the urn $U_\sigma$ proportional to the weight. 
		Let $\{\mhb\sigma i\}_{i=1}^m$ be the sequence of balls drawn from $\hat U_\sigma$, then it is easy to show that 
	\eqn{\label{prob.v}
	\mprob{\xi(\mb{\sigma+1}1)=v\,|\,U_\sigma}=\mprob{\mv{\sigma+1}1=v\,|\,G_{\sigma},x_{\sigma+1}}
	}
	and $\mprob{\xi(\mhb{\sigma+1}1)=v\,|\,\hat U_\sigma}= \mprob{\mhv{\sigma+1}1 = v\,|\,\hat G_{\sigma},\hat x_{\sigma+1}}$, for $v \in V_{\sigma+1}$. As an example we will show \eqref{prob.v} for $v\in V_{\sigma+1}\backslash\{\tau,\sigma+1\}$.
	Observe that in this case the left hand side of \eqref{prob.v} corresponds to the probability on the event that we draw the red ball numbered $v$ or one of the $d_\sigma(v)-m$ white balls, thus
	\begin{align*}
	\mprob{\xi(\mb{\sigma+1}1)=v\,|\,U_\sigma}
	&=
	\frac{(m+\delta)\mR(x_v)+(d_\sigma(v)-m)\mR(x_v)}{\mn{U_\sigma}}
	\\&=\frac{(d_\sigma(v)+\delta)\mR(x_v)}{\mn{U_\sigma}}
	=\mprob{\mv{\sigma+1}1=v\,|\,G_{\sigma},x_{\sigma+1}},
	\end{align*}
	by \eqref{rule.sel}, since $\mn{U_\sigma}=M_{\sigma,n}$ (see Remark \ref{rem:total_w}).
%

	\subsection{The joint distribution of drawing balls}	\label{sec:coupletoballs}
		In this section we describe how we simultaneously draw the balls from the urns $U_\sigma$ and $\hat U_\sigma$.
		As before, we will assume that $T_{\sigma,n} \leq \hat T_{\sigma,n}$, or, equivalently, $\mn{U_\sigma} \leq \mn{\hat U_\sigma}$, see Remark \ref{rem:total_w}.
		In the last part of this section we calculate the probability of the event $\{\mb\sigma i \not = \mhb\sigma i\}$, for $i=1,2,\ldots,m$, and $\tau \leq \sigma \leq n$,  i.e., the event that the two balls $\mb\sigma i$ and $\mhb\sigma i$ in the $i^{\rm th}$ draw do not agree on number or color, which we call a \emph{mismatch}.

			Define the following sets
					\begin{align}\nonumber
					\mX R&=U_\sigma \backslash \mhX U,& C_\sigma&=U_\sigma \cap \mhX U& 	&\text{and}& 	L_\sigma&=\mhX U \backslash \mX U,&
					\end{align}
			where, as before, we compare the balls by color and number. 
			\begin{rem}\label{rem:balls}
				By construction, $\mX L$ only contains white and orange balls, $\mX C$ contains only white, red and green balls, and $\mX R$ contains only  white, purple and blue balls. Furthermore, concerning the weights, we have the following relations
			\begin{align}\label{relCRCLL}
			\mn{\mX C}+\mn{\mX R}&=\mn{\mX U}& 
			&\text{and}&
			&\mn{\mX C}+\mn{\mX L}=\mn{\mhX U}.&
			\end{align}
			\end{rem}

				Next, we give the joint distribution of drawing balls from the urns $U_\sigma$ and $\hat U_\sigma$. 
				
				\bigskip\noindent{\bf The joint distribution:} 				
				Draw, with replacement, $m$ balls $\mb{\sigma+1}1,\ldots,\mb{\sigma+1}m$ from $\mX U$.				
				For convenience we write $\mb{}i=\mb{\sigma+1}i$ for $i=1,\ldots,m$. For each $i$, 
				we define $\mhb{}i=\mhb{\sigma+1}i$ by
				\begin{itemize}
					\item If $\mb{}i \in \mX C $ then, with probability
					\eqn{ \label{redistr.prob}
					\frac{\mn{\mX U}}{\mn{\mhX U}},
					}
					we set $\mhb{}i=\mb{}i$, otherwise we choose $\mhb{}i$ from $\mX L$, i.e., we choose $b \in \mX L$ with probability $w(b)/\mn{\mX L}$; observe that the quotient in \eqref{redistr.prob} is bounded by 1, because, as remarked earlier, $\mn{U_\sigma} \leq \mn{\hat U_\sigma}$.
				
					\item If $\mb{}i \in \mX R$, then we choose $\mhb{}i$ from $\mX L$, i.e., choose $b \in \mX L$ with probability $w_\sigma(b)/\mn{L_\sigma}$.
				\end{itemize}

	\bigskip

	\noindent{\bf The marginal distributions:} Denote by $\mcprob[]{\,\cdot\,}$ the joint probability measure under the above introduced coupling.
	Furthermore, let $\mcprob[\sigma]{\,\cdot\,}=\mcprob[]{\,\cdot\,|\mX U,\mhX U}$.
	We will show that under the coupling 
	\begin{align}\label{eq:check:marg1}
	\mcprob[\sigma]{\mb{}1=b}&=\frac{\mX w(b)}{\mn{\mX U}}=\mprob{\mb{}1=b\,|\,U_\sigma}&
	\intertext{and}\label{eq:check:marg2}
	\mcprob[\sigma]{\mhb{}1=b}&=\frac{\mX w(b)}{\mn{\mhX U}}=\mprob{\mhb{}1=b\,|\,\hat U_\sigma},&
	\end{align}
	for $b \in U_\sigma$ and $b \in \hat U_\sigma$, respectively.	
	The claim \eqref{eq:check:marg1} is true by construction. 
%
	For the claim \eqref{eq:check:marg2}, if $b \in \hat U_\sigma$, then this implies that $b\in C_\sigma$ or $b\in L_\sigma$, but not in both.
	Firstly, assume $b \in \mX C$, then
	\begin{align}
\nonumber	\mcprob[\sigma]{\mhb{}1=b}=\mcprob[\sigma]{\mhb{}1=\mb{}1|\mb{}1=b}\mcprob[\sigma]{\mb{}1=b}=\frac{\mn{\mX U}}{\mn{\mhX U}}\frac{\mX w(b)}{\mn{\mX U}}
	=\frac{\mX w(b)}{\mn{\mhX U}}.
	\end{align}	
	Secondly, if $b \in \mX L$, then
	\begin{align*}
	\mcprob[\sigma]{\mhb{}1=b}&=\mcprob[\sigma]{\mhb{}1=b|\mb{}1 \in \mX C}\mcprob[\sigma]{\mb{}1 \in \mX C}
			\\&\,\,\,\,+\mcprob[\sigma]{\mhb{}1=b|\mb{}1 \in \mX R}\mcprob[\sigma]{\mb{}1 \in \mX R}\\&=\frac{\mX w(b)}{\mn{\mX L}}\br{1-\frac{\mn{\mX U}}{\mn{\mhX U}}}\cdot\frac{\mn{\mX C}}{\mn{\mX U}}+
	\frac{\mX w(b)}{\mn{\mX L}}\cdot\frac{\mn{\mX R}}{\mn{\mX U}}
	\\&=\frac{\mX w(b)\br{\br{\mn{\mX C}+\mn{\mX R}}\mn{\mhX U}-\mn{\mX U}\mn{\mX C}}}{\mn{\mX L}\mn{\mX U}\mn{\mhX U}}
	=\frac{\mX w(b)}{\mn{\mhX U}},
	\end{align*}
	where we used in the last step the relations given by \eqref{relCRCLL}. Hence, also, the claim \eqref{eq:check:marg2} is true.

\subsection{The joint growth rule between coupled graphs}	\label{sec:pertubation}
Fix $\tau \in \{1,2,\ldots,n\}$, as before, and consider the graph process $\{G_s\}_{s=0}^{\tau-1}$. Let $\{\hat G_s\}_{s=0}^{\tau-1}$ be an identical copy of $\{G_s\}_{s=0}^{\tau-1}$, and choose at time $\tau$ the position $x_\tau$ and $\hat x_\tau$ in $G_\tau$ and $\hat G_\tau$, respectively,	at random in $S$, independently of each other.
		Using the urns, we will describe the growth of the graphs $G_\sigma$ and $\hat G_\sigma$ over time.

	At time $\tau$ we apply the {\bf Growth Rule}, independently, on the graphs $G_{\tau-1}$ and $\hat G_{\tau-1}$. Then at time $\sigma+1$, for $\sigma \geq \tau$, let $x_{\sigma+1}$ randomly chosen from $S$ and set $\hat x_{\sigma+1}=x_{\sigma+1}$. Let $U_\sigma$ and $\hat U_\sigma$ the urns correspond to $(G_\sigma, x_{\sigma+1})$ and $(\hat G_\sigma, \hat x_{\sigma+1})$, respectively.
	Note that this is precisely the setting as  described in \namedref{seq:the urns} and, as a consequence, we can use the results of \namedref{sec:coupletoballs}.
	Draw with replacement $m$ balls, $\{b_{\sigma+1}^{\smallsup{i}}\}_{i=1}^m$, from $U_\sigma$, then the vertex-endpoints of $x_{\sigma+1}$ are given $\{\xi(b_{\sigma+1}^{\smallsup{i}})\}_{i=1}^m$. We, also, draw with replacement $m$ balls, $\{\hat b_{\sigma+1}^{\smallsup{i}}\}_{i=1}^m$, from $\hat U_\sigma$, and construct $\hat G_{\sigma+1}$ in the same way.

	\subsection{The probability on a mismatch} \label{sec:mismatchballs}
	The event of a mismatch of vertex-endpoints in the graphs $G_{\sigma}$ and $\hat G_\sigma$, $\sigma \geq \tau$, can be expressed in terms of drawing balls from the urns $U_\sigma$ and $\hat U_\sigma$, since
	\eqn{\label{eventsin}
	\bra{\mv{\sigma+1}1 \not = \mhv{\sigma+1}1}	
	=
	\bra{\xi(\mb{\sigma+1}1) \not = \xi(\mhb{\sigma+1}1)}
	\subset
	\bra{\mb{\sigma+1}1 \not = \mhb{\sigma+1}1}.
	}
	Thus, we will concentrate on the probability of a mismatch between the drawn balls from the urns.
	Without loss of generality, we assumed that $\mn{U_\sigma} \leq \mn{\hat U_\sigma}$ or, equivalently, $T_{\sigma,n} \leq \hat T_{\sigma,n}$. Using the joint distribution of the urns, see Section \ref{sec:coupletoballs}, and \eqref{relCRCLL}, we obtain
	\begin{align}
				\nonumber \mcprob[\sigma]{\mb{\sigma+1}1 \not = \mhb{\sigma+1}1}
				&
				=1- \sum_{b\in\mX C}\mcprob[\sigma]{\mhb{\sigma+1}1 = \mb{\sigma+1}1| \mb{\sigma+1}1=b}\mcprob[\sigma]{\mb{\sigma+1}1=b}
				\\&=1-\sum_{b\in\mX C} \frac{\mn{\mX U}}{\mn{\mhX U}}\cdot \frac{w(b)}{\mn{\mX U}}
				=1-\frac{\mn{\mX C}}{\mn{\mhX U}}
				=\frac{\mn{\mX L}}{\mn{\mhX U}}.
				\label{P(b=hatb)}
	\end{align}
	By \eqref{def:M} and Remark \ref{rem:total_w}, we can bound the denominator on the right hand side of \eqref{P(b=hatb)} from below by
	\eqn{
		\mn{\mhX U}\geq \mn{\mX U}=M_{\sigma,n} \geq \alpha\Theta I_n \sigma.
		\nonumber
		}
	Next, we consider the numerator on the right hand side of \eqref{P(b=hatb)}. The set $\mX L$ only contains white balls and the orange ball, see Remark \ref{rem:balls}.
	 Therefore, compare \eqref{eq:weights}, the total weight of $\mX L$ can be written as
	\eqn{\nonumber \label{||L||}
	\mn{\mX L}
		=\sum_{h \in {\cal E}_\sigma} \mR(x_h)+ (\hat d_\sigma(\tau)+\delta)\mR(\hat x_\tau)
	,
	}
	where
	\eqn{\label{diffE}
		{\cal E}_{\sigma}=
		\cup_{e\in E_{\sigma}\backslash\hat E_{\sigma}}\{e\,:\,\head e\not = \tau\}.
	}
	Thus, the probability on a mismatch between balls is bounded from above by
	\eqn{ \label{eq:eq:diff1}
	\mcprob[\sigma]{\mb{\sigma+1}1 \not = \mhb{\sigma+1}1} 
			\leq \frac{\sum_{h \in {\cal E}_\sigma} \mR(x_h)+ (\hat d_\sigma(\tau)+\delta)\mR(\hat x_\tau)}{\alpha\Theta I_n \sigma} 
	}

	\begin{rem}
	If $T_{\sigma,n} > \hat T_{\sigma,n}$, then it should be clear that one can interchange the roles of $G_\sigma$ and $\hat G_\sigma$ in \namedref{sec:Coupling}, which implies that for this case
	\eqn{ \label{eq:eq:diff2}
	\mcprob[\sigma]{\mb{\sigma+1}1 \not = \mhb{\sigma+1}1} 
			= \frac{\mn{\mX R}}{\mn{\mX U}}
			\leq \frac{\sum_{h\in\hat {\cal E}_\sigma} \mR(x_h)+ (d_\sigma(\tau)+\delta)\mR(x_\tau)}{\alpha\Theta I_n \sigma}, 
	}
	where $\hat {\cal E}_{\sigma}=\cup_{e\in \hat E_{\sigma}\backslash E_{\sigma}}\{ e \,:\, {\head e\not = \tau}\}$.
	\end{rem}

	\section{Proof of the main results} \label{sec:proof.main}		
	In this section we will prove the main results, i.e, Theorem \ref{theorem:1}, \ref{thm:diameter} and \ref{thm:diameter2}.
 The diameter results, Theorem \ref{thm:diameter} and \ref{thm:diameter2}, can be proved almost immediately using the proofs in \cite{flax07}, but this is not true for Theorem \ref{theorem:1}. 
 The proof of Theorem \ref{theorem:1} relies on Lemma \ref{conc.W} and this takes more effort.

 This section is divided into 3 parts: in the first part we will give the proof of Lemma \ref{conc.W}, then, in the second part, we will give the proof of the main results, and in the  last part we show a bound on the number of expected mismatches, which is necessary for the proof of Lemma \ref{conc.W}.
 Before doing so, we will consider the number of mismatches between $G_\sigma$ and $\hat G_\sigma$, for $\sigma \geq \tau \geq 1$, where a perturbation is made at time $\tau$ as defined in Section \ref{sec:pertubation}.

 At each of the times $s=\tau,\tau+1,\ldots,\sigma-1$, we sample (with replacement) $m$ balls from each of the urns $U_s$ and $\hat U_s$. After $m$ draws we end up with the balls $\{\mb{s} i\}_{i=1}^m$ and $\{\mhb{s} i\}_{i=1}^m$.
 If the two balls $\mb s i$ and $\mhb s i$ in the $i^{\rm th}$ draw do not agree on number or color, then, as before, we call the draw a mismatch, i.e., $\{\mb s i \not = \mhb s i\}$.
 Let $\Delta^\tau_\sigma$ the total number of mismatches between the urns $U_\sigma$ and $\hat U_\sigma$, then
	\eqn{\label{totmismatches}
	\Delta_\sigma=\Delta_\sigma^\tau=\sum_{s=\tau}^\sigma \sum_{i=1}^m \indic{\mb s i \not = \mhb s i}.
	}
	Furthermore, for $u \in S$ define
	\eqn{\label{def:Delta.x}
	\Delta_\sigma(u)=\Delta^\tau_\sigma(u)=\sum_{s=\tau}^\sigma \sum_{i=1}^m \brc||{F_n(|x_{\xi( \mb s i)}-u|)-F_n(|x_{\xi(\mhb s i)}-u|)}.
	}
	Next, we will relate the expected values of \eqref{totmismatches} and \eqref{def:Delta.x}.
	Fix any $y\in S$ and let $U$ be randomly chosen in $S$, then 
		\begin{align*} \nonumber
	\mexpec{\brc||{{F_n(|x_{\xi( \mb s i)}-U|)-F_n(|x_{\xi(\mhb s i)}-U|)}}} 
		\hspace{-5cm}&\hspace{5cm}
	 \\ &\leq \mexpec{\br{{F_n(|x_{\xi( \mb s i)}-U|)+F_n(|x_{\xi(\mhb s i)}-U|)}}\indic{\mb s i \not= \mhb s i}}\nonumber
	\\ & \leq  \mexpec{2\int_SF_n(|y-u|)\,\mathrm du \indic{\mb s i \not= \mhb s i}}
	=2I_n\mexpec{\indic{\mb s i \not= \mhb s i}} ,
	\end{align*}	
	where we used \eqref{calcIf}. Thus,
	\eqn{\label{eq.shrtcut.2}
	\mexpec{\Delta_\sigma(U)} \leq 2 I_n\sum_{s=\tau}^\sigma \sum_{i=1}^m\mexpec{ \indic{\mb s i \not= \mhb s i}}
	=2I_n\mexpec{\Delta_\sigma}
	.
	}
	
	\begin{lem} \label{lem:miscouplings}
		Under the conditions of Theorem \ref{theorem:1}, let $\sigma \geq 1$ and $U$ randomly chosen in $S$, then for some constant $C>0$,
		\eqn{\label{anders1}
		\mexpec{\Delta^\tau_\sigma(U)}\leq C m I_n \br{\frac{\sigma}{\tau}}^{\frac{1}{\alpha \Theta/ m}}\log \sigma,
		}
		and, as a consequence,
		\eqn{\nonumber
		\mexpec{\Delta^\tau_\sigma(U)} \leq C m I_n \br{\frac{\sigma}{\tau}}^{\frac{2}{\alpha}},
		}
		since $(\Theta/ m)^{-1} < 2$.
	\end{lem}
	The proof of the above lemma is deferred to \S\ref{bound.mismatches}.

%
%
%

			\begin{rem}
		For the proof of the main result, we need that the number of mismatches is of $o(\sigma)$, which implies that the exponent in \eqref{anders1} should be smaller than 1, i.e., $m/\alpha\Theta<1$.
		For $\alpha>2$ and $\delta>-m$ this is indeed the case:
		\begin{align*}
		\frac{m}{\Theta}  \leq \frac{2m}{m+(m+\delta)} < \frac 2m \leq 2,
		\end{align*}
		thus $m/\alpha\Theta < 1$.
		
		If $\delta=0$, which is precisely the model introduced in \cite{flax07}, then the condition simplifies to $1/\alpha<1$, which is a weaker condition than the condition used in \cite{flax07}: $2/\alpha<1$. Nevertheless, we cannot get rid of the condition $\alpha>2$, because we need that the event ${\cal B}_\sigma$ occurs with high probability, see \eqref{eq:CA}.
	\end{rem}

			\subsection{Proof of Lemma \ref{conc.W}} \label{subs conc.W}
			
		In this section we will prove Lemma \ref{conc.W} using the Azuma-Hoeffding inequality, which provides exponential bounds for the tails of a special class of martingales:
			\begin{lem}\label{Hoeffdingvaria} 
			Let $\{X_\tau\}_{\tau\geq 0}$ be a martingale process with the property that, with probability 1, there exists a sequence of positive constants $\{e_\tau\}_{\tau \geq 1}$ such that
			$$
			\brc||{X_\tau-X_{\tau-1}} \leq e_\tau,
			$$
			for all $\tau \geq 1$. Then, for every $\lambda >0$,
			$$
			\mprob{|X_\sigma-X_0| \geq \lambda} \leq 2\exp\bra{-\frac{\lambda^2}{4\sum_{\tau=1}^\sigma e_\tau^2}}.
			$$
			\end{lem}
			For a proof of this lemma, we refer to \cite{Hoeffding:1963:PIS}.
			
			We will apply Lemma \ref{Hoeffdingvaria} by taking a Doob-type martingale \mproef{$}{$$}X_\tau=\mexpec{T_{\sigma,n}(U) \,|\, G_\tau}\mproef{$}{$$}, where $U$ is chosen at random in $S$. By, convention, we let $G_0$ be the empty graph, then
			\begin{align*}
			X_0&=\mexpec{T_{\sigma,n}(U)}& &\text{and}& X_\sigma&=T_{\sigma,n}(U).&
			\end{align*}
			At each time step $s$ we add a new vertex and $m$ edges, see the {\bf Growth Rule} in Section \ref{gpaf:sec:model}, call this an action.
			We call an action $A$ \emph{acceptable} if the action can be applied with positive probability. Furthermore, denote by ${\cal A}(G)$ the set of all \emph{acceptable actions} that can be applied on the graph $G$.

	Clearly,
		\begin{multline}\label{eq.sup.e}
	\brc||{\mexpec{T_{\sigma,n}(U)|G_{\tau-1}} - \mexpec{T_{\sigma,n}(U)|G_{\tau}}}
		\\\leq 
		\sup_{G_{\tau-1}} 
		\mathop{\sup_{A \in {\cal A}(G_{\tau-1})}}_{\hat A \in {\cal A}(\hat G_{\tau-1})}
		\brc||{\mexpec{T_{\sigma,n}(U)|G_{\tau-1}(A)} - \mexpec{T_{\sigma,n}(U)|G_{\tau-1}(\hat A)}},
	\end{multline}
	where the first supremum is taken over all possible graphs $G_{\tau-1}$.
	Next, fix the graph $G_{\tau-1}$ and let $G_\tau=G_{\tau-1}(A)$ be the graph by applying the action $A$ on the graph $G_{\tau-1}$. Similarly, define $\hat G_\tau=G_{\tau-1}(\hat A)$. Thus, one can rewrite the right hand side of \eqref{eq.sup.e} as
	$$
	e_\tau = \sup_{G_{\tau-1}} 
	\mathop{\sup_{A \in {\cal A}(G_{\tau-1})}}_{\hat A \in {\cal A}(\hat G_{\tau-1})}
	\brc||{\mexpec{T_{\sigma,n}(U)|G_\tau} - \mexpec{\hat T_{\sigma,n}(U)|\hat G_\tau}}
	.$$	
	Using the triangle inequality, the above implies, under the coupling,
	\eqn{\label{P:max:gr}
	e_\tau \leq \sup_{G_{\tau-1}}
	\mathop{\sup_{A \in {\cal A}(G_{\tau-1})}}_{\hat A \in {\cal A}(\hat G_{\tau-1})}
	\mcexpec[]{\brc||{T_{\sigma,n}(U) -\hat T_{\sigma,n}(U)} } . 
	}

	We claim that,  independently of $G_{\tau-1}, A$ and $\hat A$, and, for $\sigma \geq \tau$,
	\eqn{\label{das:sing}
	\mcexpec[]{\brc||{T_{\sigma,n}(U) -\hat T_{\sigma,n}(U)} }
	\leq \tilde C m I_n(\sigma/\tau)^{2/\alpha},
	}
	where the proof of this claim is deferred to the end of this section.
	Thus, using \eqref{P:max:gr} and \eqref{das:sing},
	\eqn{ \nonumber 
	\sum_{\tau=1}^\sigma e_\tau^2 \leq  \tilde C^2 m^2 I_n^2 \sigma^{4/\alpha}\sum_{\tau=1}^\sigma \tau^{-4/\alpha}=\mO{m ^2I_n^2(\sigma^{4/\alpha}+\sigma\log \sigma)}.
	}
	To show the above, let $\beta=4/\alpha$, then $\beta \in(0,2)$. If $\beta \in (1,2)$, then $\sum_{\tau=1}^\sigma \tau^{-\beta} < \infty$, and, if $\beta\in(0,1]$, then 
	$$
	\sigma^{\beta}\sum_{\tau=1}^\sigma \tau^{-\beta} \leq 
	\sigma^{\beta}+\sum_{\tau=2}^\sigma (\tau/\sigma)^{-\beta}
	\leq  \sigma^{\beta}+\sigma\int_{x=1}^{\sigma}x^{-1}\,\mathrm{d}x
	=\sigma^{\beta}+\sigma \log \sigma.
	$$
	From Lemma \ref{Hoeffdingvaria} we then obtain for some constant $C_1$,
	\eqn{\nonumber
	\mprob{|T_{\sigma,n}(U)-\mexpec{T_{\sigma,n}(U)}| \geq C_1 m I_n  (\sigma^{2/\alpha}+\sigma^{1/2}\log \sigma)(\log n)^{1/2}}
	\leq 2e^{-2\log n} 
	.
	}
	By taking $n$ sufficiently large, we can replace $C_1(\log n)^{1/2}$ by $\log n$, which is, precisely, the statement of Lemma \ref{conc.W}, given the claim \eqref{das:sing}. \qed

	\bigskip\noindent{\bf Proof of claim \eqref{das:sing}:} 
	Denote by $d^-_\sigma(v)$ the in-degree of the vertex $v$ at time $\sigma$, and observe that
	\eqn{\label{P:din}
	d_\sigma(v)=d^-_\sigma(v)+m,
	}
	since each vertex has by construction $m$ edges pointing outward to other vertices or itself. Furthermore, we denote by  $\my s i$ the position of the  $i^{\mathrm th}$ vertex-endpoint at time $s$, thus
	\eqn{\label{P:my}
	\my s i=x_{\mv s i}.
	}
	By construction of $G_\tau$ and $\hat G_\tau$, we can apply the coupling introduced in \namedref{sec:pertubation}.
	Rewrite $T_{\sigma,n}(U)$, see \eqref{eq:Tt}, using \eqref{P:din}, \eqref{P:my}, and the coupling, as
	\begin{align}
		T_{\sigma,n}(U)&\nonumber
					=\sum_{v=1}^\sigma d_\sigma(v)F_n(|x_v-U|) + \delta\sum_{v=1}^\sigma F_n(|x_v-U|)
					\nonumber
			\\&\nonumber=\sum_{v=1}^\sigma d^-_\sigma(v)F_n(|x_v-U|) + (m+\delta)\sum_{v=1}^\sigma F_n(|x_v-U|)
					\\\nonumber& =	\sum_{s=1}^\sigma\sum_{i=1}^m F_n(|\my s i-U|) +(m+\delta)\sum_{v=1}^\sigma F_n(|x_v-U|)
			 \\&=	\sum_{s=1}^\sigma\sum_{i=1}^m F_n(|x_{\xi(\mb s i)}-U|) +(m+\delta)\sum_{v=1}^\sigma F_n(|x_v-U|).
			 \nonumber
	\end{align}

	Up to and including time $\tau-1$ both graphs are identical, thus the absolute difference 
	$$
	\brc||{T_{\sigma,n}(U)-\hat T_{\sigma,n}(U)},
	$$ equals to
		\begin{multline*} 
	\brc|.{\sum_{s=\tau}^\sigma\sum_{i=1}^m \br{F_n(|x_{\xi(\mb s i)}-U|)-F_n(|x_{\xi(\mhb s i)}-U|)}}\\\brc.|{\phantom{\sum_{s=\tau}^\sigma\sum_{i=1}^m}+(m+\delta) \br{F_n(|x_\tau-U|)-F_n(|\hat x_\tau-U|)}}.
		\end{multline*}	
	 Using the triangle inequality and \eqref{def:Delta.x}, we obtain
	$$
	\brc||{T_{\sigma,n}(U)-\hat T_{\sigma,n}(U)}\leq \Delta_\sigma(U)+(m+\delta)\br{F_n(|x_\tau-U|)+F_n(|\hat x_\tau-U|)}.
	$$	
	Taking expectations on both sides of the above display, and using \eqref{eq.shrtcut.0} and \eqref{eq.shrtcut.2}, yields
	\eqn{\label{Dler}
	\mcexpec[]{|T_{\sigma,n}(U)-\hat T_{\sigma,n}(U)|} 
		\leq 2I_n(\mcexpec[]{\Delta_\sigma}+(m+\delta)).
	}
	Applying Lemma \ref{lem:miscouplings} on \eqref{Dler} finally results in
	$$
	\mcexpec[]{|T_{\sigma,n}(U)-\hat T_{\sigma,n}(U)|} 
		\leq \tilde C m I_n (\sigma/\tau)^{2/\alpha},
	$$ 
	for some constant $\tilde C$. This is precisely the claim \eqref{das:sing}.
	\qed

	\subsection{Proof of the main results}
	
	In this section we show the main results. 
	The proof of Theorem \ref{theorem:1} is almost similar to the proof of Lemma \ref{conc.W}. The diameter results, i.e, Theorem \ref{thm:diameter} and \ref{thm:diameter2} will be proved by using the proofs in \cite{flax07}.
	
	\bigskip
	
	\noindent {\bf Proof of Theorem \ref{theorem:1}:}
	The first part of Theorem \ref{theorem:1}, i.e., claim \eqref{thm:claim1}, has been proved in Section \ref{sec:recurrence}. For the second part, i.e., claim \eqref{thm:claim2}, we now give a proof, which is similar to the proof of Lemma  \ref{conc.W}. Therefore, we follow the proof of the previous Section \ref{subs conc.W}, where we now choose $X_\tau=\mexpec{N_k(n)\,|\,G_\tau}$ instead of $X_\tau=\mexpec{T_{\sigma,n}(U) \,|\, G_\tau}$. Similar to \eqref{eq.sup.e}, we have that
	\begin{multline*}
	\brc||{\mexpec{N_k(n)|G_{\tau-1}} - \mexpec{N_k(n)|G_{\tau}}}
		\\\leq 
		\sup_{G_{\tau-1}} 
		\mathop{\sup_{A \in {\cal A}(G_{\tau-1})}}_{\hat A \in {\cal A}(\hat G_{\tau-1})}
		\brc||{\mexpec{N_k(n)|G_{\tau-1}(A)} - \mexpec{N_k(n)|G_{\tau-1}(\hat A)}}.
	\end{multline*}
	Using the coupling, we can bound the right hand side in the above display by twice the number of mismatches, since each mismatch can influence at most two edges.
	Thus,
	$$
	\brc||{\mexpec{N_k(n)|G_{\tau-1}} - \mexpec{N_k(n)|G_{\tau}}} \leq 2\mcexpec[]{\Delta^\tau_n}.
	$$
	Therefore, we can take $e_\tau=2\mcexpec[]{\Delta^\tau_n}$ and we, again, can apply Lemma \ref{Hoeffdingvaria}, as done in the previous section, which proves claim \eqref{thm:claim2} and hence Theorem \ref{theorem:1}.\qed
	
	\bigskip
	
	\noindent {\bf Proof of Theorem \ref{thm:diameter}:}
		The proof is almost identical to the proof of \cite[Theorem 2]{flax07}. 
		To apply this proof for general $\delta > -m$, we only need to replace the constant $c_3$ in \cite{flax07} by $c^*_3$, where $c^*_3=c_3/2$ and $c_3$ is the constant of condition {\bf S3}, see Section \ref{seq:heuristics}. This will be explained in more detail now.
		
		Pick $\mu$ and $\rho_n=\rho_n(\mu, F_n)$ such that $F_n$ is smooth for $\mu$, see conditions {\bf S1}, {\bf S2} and {\bf S3}, see Section \ref{seq:heuristics}.
		Fix $u \in S$ and denote by $A_{\rho_n}$ the spherical cap with center $u$ and radius $\rho_n$, then there exists positive constants $c_1$ and $c_2$, independent of $\rho_n$, such that
		\eqn{\label{P:sdf:An}
		A_{\rho_n}=\int_{\{w \in S\,:\, |w-u| \leq \rho_n\}} \,\mathrm d w\in [c_1\rho_n^2,c_2\rho_n^2],
		}		
		which is shown in \cite{flax07}. 
		Furthermore, in \cite{flax07} the authors consider the graph at certain time steps $t_s$, where $s$ is a positive integer, such that the area of the spherical cap is given by 
		\eqn{\label{P:sdf:An2}
		s/2 \leq A_{\rho_n}t_s \leq 3s/2.
		}
		In the proof of \cite[Theorem 2]{flax07}, the essential step is the statement
		that the probability that $v_{t_s}$ chooses vertex $v \in V_{t_s}$, assuming that $|x_{t_s} -x_v| \leq 2\rho_n$, is at least $\frac{2c_1c_3}{\alpha s}$, i.e.,
		\eqn{\nonumber 
		\mprob{v^{\smallsup{1}}_{t_s}=v \,|\,G_{t_s-1}} \geq \frac{2c_1c_3}{\alpha s}.
		}
		In our model this is still true, when we replace $c_3$ by $c^*_3=c_3/2$, since, using the assumptions {\bf S1}, {\bf S2} and {\bf S3}, \eqref{P:sdf:An} and \eqref{P:sdf:An2},
		\begin{align*}
		\mprob{v^{\smallsup{1}}_{t_s}=v \,|\,G_{t_s-1}} &\geq \frac{(d_{t_s}(v)+\delta)F_n(|x_{t_s}-x_v|)}{\alpha\Theta I_n t_s}
		\geq \frac{(m+\delta)F_n(2\rho_n)}{\alpha \Theta I_n t_s }
		\\&		
		\geq \frac{2(m+\delta)A_{\rho_n}F_n(2\rho_n)}{\alpha \Theta I_n s }
		\geq \frac{2(m+\delta)c_1\rho_n^2F_n(2\rho_n)}{\alpha \Theta I_n s }
				\\&
				\geq \frac{2(m+\delta)c_1c_3}{\alpha \Theta s } = \frac{(m+\delta)}{\Theta } \frac{2c_1c_3}{\alpha s } > \frac{1}{2}\frac{2c_1c_3}{\alpha  s }
		=\frac{2c_1c^*_3}{\alpha s},		
		\end{align*}
		where we used that $(m+\delta)/\Theta=1+\delta/(2m+\delta) >1/2$ for $-m < \delta \leq 0$ and $(m+\delta)/\Theta \geq 1 > 1/2$ for $\delta>0$.
		If we replace the constant $c_3$ by $c_3^*$ in the proof of Theorem 2, then the proof of \cite{flax07} holds without further modifications.
		\qed

	\bigskip
	
	\noindent {\bf Proof of Theorem \ref{thm:diameter2}:}
	For $\delta=0$ the proof is given by the proof of Theorem 3 in \cite{flax07}. 
	The constant $\lambda=C_1/C_2$ in the proof of \cite[Theorem 3]{flax07} should be replaced by $\lambda=(C_1+\delta)/2C_2$, then the proof holds verbatim. 
	\qed

	\subsection{Bounding the expected number of mismatches} \label{bound.mismatches}
	In this section we will prove Lemma \ref{lem:miscouplings}.
	In the proof of the lemma, we rely on two claims, which will be stated now.
	The first claim bounds for any vertex and all time steps the expected degree:
	\eqn{\label{claim1}
		\mcexpec[]{d_v(\sigma)+\delta} \leq m C \br{\frac{\sigma}{v}}^a,
	}	
	where $C$ is some constant and
	\eqn{\label{def:a}
	a=m/\alpha\Theta.
	}
	The second claim is a technical one, which bounds the expectation of
		\renewcommand{\mcprob}[2][\sigma]{\hat{\mathbb P}_{\hspace{-.15em} #1} \hspace{-.20em} \br{#2}}
	\eqn{\label{Qdef}	
	Q_\sigma=\sum_{h \in {\cal E}_\sigma} \mR(x_h)\indic{T_{\sigma,n} \leq \hat T_{\sigma,n}}+\sum_{h \in \hat {\cal E}_\sigma} \mR(\hat x_h)\indic{T_{\sigma,n} > \hat T_{\sigma,n}}
	}
	from above. More precisely, for any $\sigma \geq \tau$,
	\eqn{\label{claim2}
		\mcexpec[]{Q_\sigma} \leq I_n\br{\mcexpec[]{\Delta_\sigma}+\mcexpec[]{d_\sigma(\tau)+\delta}+\mcexpec[]{\hat d_\sigma(\tau)+\delta}+2\Theta}.
	}
	Next, we will assume that the claims \eqref{claim1} and \eqref{claim2} do hold and we will show that Lemma \ref{lem:miscouplings} follows from these two claims. After the proof of Lemma \ref{lem:miscouplings}, we will prove both claims separately.

	\bigskip
		
	\noindent {\bf Proof of Lemma \ref{lem:miscouplings}:}	
	Let $\tau < \sigma \leq t$, then the number of mismatches is recursively defined as
	$$
	\Delta _{\sigma+1}=\Delta _{\sigma}+\sum_{i=1}^m \indic{\mb{\sigma+1}i \not = \mhb{\sigma+1}i},
	$$
	and, therefore, 
	\eqn{ \label{exp.rec}
	\mcexpec[]{\Delta _{\sigma+1}}=\mcexpec[]{\Delta _{\sigma}}+m\mcprob[]{\mb{\sigma+1}1 \not = \mhb{\sigma+1}1},
	}	
	since we draw the balls with replacement.
	
	Combining \eqref{eq:eq:diff1} and \eqref{eq:eq:diff2}, yields
	\eqn{\label{P:hfddsfds}
	\mcprob[\sigma]{\mb{\sigma+1}1 \not = \mhb{\sigma+1}1} \leq \frac{Q_\sigma+(d_\sigma(\tau)+\delta)\mR(x_\tau)+(\hat d_\sigma(\tau)+\delta)\mR(\hat x_\tau)}{\alpha\Theta I_n\sigma}.
	}
	Observe from \eqref{calcIf} that
	\begin{align*}
	\mcexpec[]{d_\sigma(\tau)\mR(x_\tau) \,|\, d_\sigma(\tau), x_\tau}&=
	d_\sigma(\tau)\mcexpec[]{F(|x_\tau-x_{\sigma+1}) \,|\, d_\sigma(\tau), x_\tau}
	\\&=d_\sigma(\tau)\int_S F(|x_\tau-u|)du=d_\sigma(\tau)I_n,
	\end{align*}
	and, hence
	$$
	\mcexpec[]{(d_\sigma(\tau)+\delta)\mR(x_\tau)}=\mcexpec[]{d_\sigma(\tau)+\delta}I_n.
	$$
	Similarly, 	$\mcexpec[]{(\hat d_\sigma(\tau)+\delta)\mR(\hat x_\tau)}=\mcexpec[]{\hat d_\sigma(\tau)+\delta}I_n$.
	Thus, taking expectations on both sides of \eqref{P:hfddsfds}, and using \eqref{claim2}, results in
	\begin{align} \label{P;fdssdafAS}
	\mcprob[]{\mb{\sigma+1}1 \not = \mhb{\sigma+1}1} &\leq \frac{\mcexpec[]{Q_\sigma}+(\mcexpec[]{d_\sigma(\tau)+\delta}I_n) +(\mcexpec[]{ \hat d_\sigma(\tau)+\delta}I_n)}{\alpha\Theta I_n\sigma},
	\end{align}
	Substituting \eqref{claim1} and \eqref{claim2} in \eqref{P;fdssdafAS}, yields
	\begin{align}
	\nonumber
	\mcprob[]{\mb{\sigma+1}1 \not = \mhb{\sigma+1}1} \leq \frac{\mcexpec[]{\Delta_\sigma}+\br{2\Theta+4mC(\sigma/\tau)^a}}{\alpha\Theta\sigma},
	\end{align}
	for some sufficiently large constant $C>0$.
	Therefore, we can bound the right hand side of equation \eqref{exp.rec} by,
	$$
	\mcexpec[]{\Delta_{\sigma+1}} 
		\leq \mcexpec[]{\Delta_{\sigma}} \br{1+\frac{a}{\sigma}} + \frac{2\Theta+4mC(\sigma/\tau)^a}{\alpha\Theta\sigma}
		=\mcexpec[]{\Delta_{\sigma}} \br{1+\frac{a}{\sigma}} + \frac{1+4C(\sigma/\tau)^a}{\sigma},
	$$
	where we used that $a=m/\alpha \Theta$ \eqref{def:a}, $1/\alpha \leq 1/2$ and $m/\Theta \leq 2$.	
	Finally, by taking the constant $C$ larger, we can replace the above inequality by
	$$
	\mcexpec[]{\Delta_{\sigma+1}} 
		\leq \mcexpec[]{\Delta_{\sigma}} \br{1+\frac{a}{\sigma}} + C \tau^{-a}\sigma^{a-1}.
	$$		
	We will now prove an upper bound for $\mexpec{\Delta_{\sigma+1}}$. To this end, we consider the solution of the recurrence relation  $q(\sigma+1)=q(\sigma)(1+a/\sigma)+b(\sigma)$, for $\sigma >\tau$, with initial condition $q(\tau)=c$. The solution is given by 
	$$
	q(\sigma)=\frac{\Gamma(\sigma+a)}{\Gamma(\sigma)}\sum_{s=\tau}^{\sigma-1} \frac{b(s)\Gamma(s+1)}{\Gamma(s+1+a)}+c\frac{\Gamma(\sigma+a)\Gamma(\tau)}{\Gamma(\sigma)\Gamma(\tau+a)}
.
	$$
		The above solution implies, that if one takes $b(s)=C\tau^{-a}s^{a-1} \leq C\tau^{-a}\frac{\Gamma(s+a)}{\Gamma(s+1)}$, then
		$$
	\mcexpec[]{\Delta_\sigma} \leq q(\sigma)=	\frac{\Gamma(\sigma+a)}{\Gamma(\sigma)}\sum_{s=\tau}^{\sigma-1} \frac{b(s)\Gamma(s+1)}{\Gamma(s+1+a)}
		\leq C\tau^{-a}\frac{\Gamma(\sigma+a)}{\Gamma(\sigma)}\sum_{s=\tau}^{\sigma-1} \frac{\Gamma(s+a)}{\Gamma(s+1+a)}
		.
		$$		
The summation can be bounded from above by $\log \sigma$, therefore, for some constant $C$,
$$
\mcexpec[]{\Delta_\sigma} 
		\leq C\br{\frac{\sigma}{\tau}}^a \log \sigma.
$$
This proves \eqref{anders1} and hence Lemma \ref{lem:miscouplings}, given the claims  \eqref{claim1} and \eqref{claim2}.
	\qed

	\bigskip\noindent{\bf Proof of \eqref{claim1}:}		
		Note that, see \eqref{def:M} and \eqref{rule.sel}, 
	\begin{align*}
		\mcexpec[\sigma]{d_{\sigma+1}(v)+\delta}		\nonumber	&=(d_{\sigma}(v)+\delta)+\mcexpec[\sigma]{d_{\sigma+1}(v)-d_{\sigma}(v)}
		\\&={(d_{\sigma}(v)+\delta)\br{1+m\frac{A_{\sigma,n}(x_v)}{M_{\sigma,n}(x_{\sigma+1})}}} 
			\mproef{}{\\&}\leq {(d_{\sigma}(v)+\delta)\br{1+m\frac{A_{\sigma,n}(x_v)}{\alpha \Theta I_n \sigma}}}. 
	\end{align*}
	Therefore, by taking expectations on both sides in the above display, and using \eqref{eq.shrtcut.0}, the value of $\mcexpec[]{d_{\sigma+1}(v)+\delta}$ is bound from above by
	\begin{align}\nonumber	
		\mcexpec[]{(d_{\sigma}(v)+\delta)\br{1+m\frac{\mcexpec[]{A_{\sigma,n}(x_v)\,|\,G_{\sigma}}}{\alpha\Theta I_n \sigma}}}
			=\br{1+\frac{m}{\alpha\Theta\sigma}}\mcexpec[]{(d_v(\sigma)+\delta)}.
	\end{align}
	Thus, by induction, and using $a=m/\alpha\Theta$,
		\begin{align*}
	\mcexpec[]{d_{\sigma+1}(v)+\delta}&
				=\br{\frac{\sigma+a}{\sigma}}\mcexpec[]{(d_v(\sigma)+\delta)}			
				=\frac{\Gamma(\sigma+a+1)\Gamma(v)}{\Gamma(\sigma+1)\Gamma(v+a)}\mcexpec[]{(d_v(v)+\delta)}.
	\end{align*}
	Finally, note that $d_v(v) \leq 2m$ and that $\delta$ is a constant, which implies the claim \eqref{claim1}. \qed
	
	\bigskip\noindent{\bf Proof of \eqref{claim2}:}		
	Observe that for $i=1,\ldots,m$ and $s\geq\tau$, see \eqref{P:my},
	\eqn{\label{eq:rel.v.y}
	\indic{\mv si \not = \mhv si} + \indic{\mv si = \mhv si=\tau} = \indic{\my si \not = \mhy si}=\indic{x_{\mv s i} \not =\, \hat x_{\mhv s i}},
	}
	since, $x_i=\hat x_i$ if $i\not =\tau$ and $x_\tau \not = \hat x_\tau$. The above statement is stronger than \eqref{eventsin}.
	Using the definition of ${\cal E}_\sigma$, see \eqref{diffE}, and \eqref{eq:rel.v.y}, we have that
	$$
	\sum_{h \in {\cal E}_\sigma} \mR(x_h)
			\leq \sum_{s=\tau}^\sigma\sum_{i=1}^m \mR(\my si )\indic{\my si \not = \mhy si},
	$$
	and we can bound $\sum_{h \in \hat{\cal E}_\sigma} \mR(x_h)$ in a similar way. 
	Hence, we can bound \eqref{Qdef} from above by	
	\begin{align}	
	Q_\sigma \leq \nonumber
			&\sum_{s=\tau}^\sigma\sum_{i=1}^m \indic{\my si \not = \mhy si} \Bigl(\mR(\my si )\indic{T_{\sigma,n}) \leq \hat T_{\sigma,n}} + \mR(\mhy si )\indic{T_{\sigma,n} > \hat T_{\sigma,n}}\Bigr)
			\\&=\sum_{s=\tau}^\sigma\sum_{i=1}^m \indic{\my si \not = \mhy si} \mR(\mhy si )
			\nonumber\\&\hspace{.5cm}+
			\indic{T_{\sigma,n} \leq \hat T_{\sigma,n}}\sum_{s=\tau}^\sigma\sum_{i=1}^m \indic{\my si \not = \mhy si}\br{\mR(\my si ) - \mR(\mhy si )}. \label{finSumd}
	\end{align}
		Next, we will show that the rightmost double sum of \eqref{finSumd} can be bounded by 
		$$
		(m+\delta)(\mR(\hat x_\tau)-\mR(x_\tau)).
		$$					
		For this, we rewrite $T_{\sigma,n}$ as 
		\begin{align} \nonumber
			T_{\sigma,n}=\sum_{s=1}^\sigma (d_s(\sigma)+\delta)\mR(x_s)
					=\sum_{s=1}^\sigma d^-_s(\sigma) \mR(x_s) 
								+ (m+\delta)\sum_{s=1}^\sigma \mR(x_s).
		\end{align}
		Note that
		\begin{align}
		\sum_{s=1}^\sigma d^-_s(\sigma) \mR(x_s) \nonumber
		&=\sum_{s=1}^\sigma \bra{\sum_{t=s}^\sigma\sum_{i=1}^m \indic{\mv t i =s}} \mR(x_s)
		\\&=\sum_{t=1}^\sigma \sum_{i=1}^m \bra{\sum_{s=1}^t \indic{\mv t i =s} \mR(x_s)}
		=\sum_{t=1}^\sigma \sum_{i=1}^m \mR(\my t i)\nonumber.
		\end{align}
		Therefore,
		\eqn{\nonumber
		T_{\sigma,n}=\sum_{s=1}^\sigma \sum_{i=1}^m \mR(\my si)+ (m+\delta)\sum_{s=1}^\sigma \mR(x_s)
		}
		and a similar result hold for $\hat T_{\sigma,n}$.
		The difference of these two expressions equals:
		$$
		T_{\sigma,n}-\hat T_{\sigma,n} 
		=\sum_{s=1}^\sigma \sum_{i=1}^m (\mR(\my si)-\mR(\mhy si))+  (m+\delta)(\mR(x_\tau)-\mR(\hat x_\tau)),
		$$
		or
		$$
		\sum_{s=1}^\sigma \sum_{i=1}^m (\mR(\my si)-\mR(\mhy si))=T_{\sigma,n}-\hat T_{\sigma,n}-(m+\delta)(\mR(x_\tau)-\mR(\hat x_\tau)).
		$$
		Hence,
		\eqn{\label{finSumdR}
		\indic{T_{\sigma,n} \leq \hat T_{\sigma,n}}\sum_{s=1}^\sigma \sum_{i=1}^m (\mR(\my si)+\mR(\mhy si))
		\leq (m+\delta)(\mR(x_\tau)-\mR(\hat x_\tau)).
		}
		Substituting \eqref{finSumdR} in \eqref{finSumd} yields,
		\begin{align}
		Q_\sigma \nonumber
			\leq \sum_{s=\tau}^\sigma\sum_{i=1}^m \indic{\my si \not = \mhy si} \mR(\mhy si )
			+
			(m+\delta)(\mR( x_\tau)+\mR(\hat x_\tau)).
		\end{align}
		Taking the conditional expectation with respect to the graphs $G_\sigma$ and $\hat G_\sigma$ results in
		\begin{multline}
		\mcexpec[]{Q_\sigma\,|\,G_\sigma,\hat G_\sigma}	
			= \sum_{s=\tau}^\sigma\sum_{i=1}^m \indic{\my si \not = \mhy si} \mcexpec[]{\mR(\mhy si )\,|\,G_\sigma,\hat G_\sigma}
			\\+
			(m+\delta)\br{\mcexpec[]{\mR(x_\tau)\,|\,G_\sigma,\hat G_\sigma}+\mcexpec[]{\mR(\hat x_\tau)\,|\,G_\sigma,\hat G_\sigma}}.\label{P:fdsdsfaasd}
		\end{multline}
		For any fixed value $x\in S$ and using \eqref{calcIf}, we have that
		$$
		\mcexpec[]{\mR(x)\,|\,G_\sigma,\hat G_\sigma}=\mcexpec[]{F(|x-x_{\sigma+1}|)\,|\,G_\sigma,\hat G_\sigma}=I_n,
		$$
		which, in turns, yields that \eqref{P:fdsdsfaasd} can be rewritten as
		$$
		\mcexpec[]{Q_\sigma\,|\,G_\sigma,\hat G_\sigma}
			= I_n\br{\sum_{s=\tau}^\sigma\sum_{i=1}^m \indic{\my si \not = \mhy si}+2(m+\delta)}
			.
		$$
		Note that \eqref{eq:rel.v.y} implies
		$$
		\sum_{s=\tau}^\sigma\sum_{i=1}^m \indic{\my si \not = \mhy si}
		=\sum_{s=\tau}^\sigma\sum_{i=1}^m \indic{\mv si \not = \mhv si}
		+\sum_{s=\tau}^\sigma\sum_{i=1}^m \indic{\mv si = \mhv si=\tau}
		\leq \Delta_\sigma+d_\sigma(\tau).
		$$
		Thus,
		$$
		\mcexpec[]{Q_\sigma\,|\,G_\sigma,\hat G_\sigma} \leq I_n\Bigl
(\Delta_\sigma+d_\sigma(\tau)+2(m+\delta)\Bigr)
	=I_n\Bigl(\Delta_\sigma+(d_\sigma(\tau)+\delta)+(2m+\delta)\Bigr).
		$$
		For convenience, we will use the following weaker statement:
		$$
		\mcexpec[]{Q_\sigma\,|\,G_\sigma,\hat G_\sigma} \leq 
		I_n(\Delta_\sigma+(d_\sigma(\tau)+\delta)+(\hat d_\sigma(\tau)+\delta)+2\Theta),
		$$
		where we replaced $(2m+\delta)$ by $2\Theta$, see \eqref{eq:def:theta}.
		Finally, by taking the expectation on both sides  in the above display, we obtain the claim \eqref{claim2}.
	\qed

 \newcommand{\shortversion}[2]{\textcolor{blue}{#1}\textcolor{red}{[[#2]]}}
	\newcommand{\longversion}[1]{\textcolor{red}{#1}}
	
		\bigskip \noindent {\bf Acknowledgments: }  
		The work of HvdE was supported in part by Netherlands Organisation for Scientific Research (NWO). I thank Gerard Hooghiemstra for carefully reading the manuscript and many suggestions, which resulted in a more accessible paper.

  \bibliographystyle{apt}

\end{document}